\documentclass[smallextended, nospthms]{svjour3}
\usepackage[utf8]{inputenc}
\usepackage[T1]{fontenc}
\usepackage{amsmath, amsthm, amsfonts}
\usepackage{ amssymb }

\usepackage{graphicx}
\usepackage{placeins}
\usepackage{xcolor}
\usepackage{hyperref}
\hypersetup{
    colorlinks = true,
    citecolor=magenta,
    filecolor=red,
    linkcolor=blue,
    urlcolor=red
}
\usepackage{amsrefs}

\newcommand{\Z}{\mathbb{Z}}
\newcommand{\R}{\mathbb{R}}

\newcommand{\Ms}{\mathcal{M}}
\newcommand{\Ss}{\mathcal{S}}
\newcommand{\Ssp}{\mathring{S}_+}

\newcommand{\BB}{\beta}
\newcommand{\ZZ}{\zeta}
\newcommand{\WW}{\omega}

\newcommand*{\SL}{\mathsf{SL}}
\newcommand*{\SO}{\mathsf{SO}}

\DeclareMathOperator{\tr}{tr}

\renewcommand{\sl}{\mathfrak{sl}}
\newcommand*{\so}{\mathfrak{so}}
\newcommand*{\sff}{\mathrm{II}}

\newcounter{thm}
\numberwithin{thm}{section}

\theoremstyle{plain}
\newtheorem{theorem}[thm]{Theorem}
\newtheorem{lemma}[thm]{Lemma}
\newtheorem{ass}[thm]{Assumption}
\newtheorem{proposition}[thm]{Proposition}
\newtheorem{corollary}[thm]{Corollary}
\newtheorem{definition}[thm]{Definition}
\newtheorem{notation}[thm]{Notation}
\newtheorem{conjecture}[thm]{Conjecture}
\newtheorem{question}[thm]{Question}

\theoremstyle{remark}
\newtheorem{rem}[thm]{Remark}
\newtheorem{remark}[thm]{Remark}

\newtheorem{ex}[thm]{Example}

\title{Geodesic motion on $\SL(n)$ with the Hilbert-Schmidt metric}

\author{Audrey Rosevear \and Samuel Sottile \and Willie WY Wong}
\institute{Audrey Rosevear \at Amherst College, Amherst, Massachusetts, USA \\\email{arosevear22@amherst.edu} \and %
	Samuel Sottile \at Michigan State University, East Lansing, Michigan, USA \\\email{sottile1@msu.edu} \and %
Willie WY Wong (corresponding author)\at Department of Mathematics, Michigan State University, East Lansing, Michigan, USA \\\email{wongwwy@math.msu.edu}}
\begin{document}

\maketitle

\begin{abstract}
    We study the geometry of geodesics on $\SL(n)$, equipped with the Hilbert-Schmidt metric which makes it a Riemannian manifold. These geodesics are known to be related to affine motions of incompressible ideal fluids. The $n = 2$ case is special and completely integrable, and the geodesic motion was completely described by Roberts, Shkoller, and Sideris; when $n = 3$, Sideris demonstrated some interesting features of the dynamics and analyzed several classes of explicit solutions. Our analysis shows that the geodesics in higher dimensions exhibit much more complex dynamics. We generalize the Virial-identity-based criterion of unboundedness of geodesic given by Sideris, and use it to give an alternative proof of the classification of geodesics in 2D obtained by Roberts--Shkoller--Sideris.
    We study several explicit families of solutions in general dimensions that generalize those found by Sideris in 3D. We additionally classify all ``exponential type'' geodesics, and use it to demonstrate the existence of purely rotational solutions whenever $n$ is even. 
    Finally, we study ``block diagonal'' solutions using a new formulation of the geodesic equation in first order form, that isolates the conserved angular momenta. 
    This reveals the existence of a rich family of bounded geodesic motions in even dimension $n \geq 4$. 
    This in particular allows us to conclude that the generalization of the swirling and shear flows of Sideris to even dimensions $n \geq 4$ are in fact dynamically unstable. 

\end{abstract}

\section*{Declarations}
\subsection*{Funding}
Research for this project was supported in part through the SURIEM program at Michigan State University, which was funded by NSA Award \#H98230-20-1-0006 and NSF Award \#1852066. 
WW also acknowledges support through a Collaboration Grant from the Simons Foundation, \#585199.
\subsection*{Conflicts of Interest} The authors have no relevant financial or non-financial interests to disclose.
\subsection*{Availability of Data and Material} Not applicable
\subsection*{Code availability} Figures for this manuscript were created using Python with the numpy, scipy, and matplotlib packages. Numerical integration performed using scipy's odeint. Exact code available upon request.

\newpage

\section{Introduction}
This paper analyzes properties of geodesics on $\SL(n)$ equipped with the Hilbert-Schmidt metric induced from the Hilbert-Schmidt inner product
\[\langle A, B\rangle = \tr(AB^T).\]
Our motivation comes from a result of Sideris \cite{AffineFluid3D} demonstrating a correspondence between geodesics on $\SL(n)$ under this metric and flows of a ball of incompressible ideal fluid which have a deformation given by a curve $A(t) \in \SL(n)$.

The case $n = 2$ is an integrable system which is discussed in depth in \cite{AffineFluid2D}. The symmetries come from conservation of energy and the left and right isometries generated by $\SO(n)$. 

In higher dimensions, these symmetries are not enough to fully integrate the system. Instead, we only obtain partial results. We see a much larger zoo of behavior, including periodic solutions which are not fully rotational (i.e. do not lie entirely in $\SO(n)$) and chaotic solutions both bounded and unbounded.

We begin in Section \ref{sect:dg} by discussing the differential geometry of $\SL(n)$.
We begin by reviewing some basic linear algebra facts and performing explicit computations of standard differential geometric objects of $\SL(n)$ as an embedded submanifold in the space of matrices with respect to the Hilbert-Schmidt metric. 
These include representations of isometries and the geodesic equation, the second fundamental form of the embedding, and formulae for the Riemann and sectional curvatures. 
Here we observe a striking difference between the $n = 2$ and $n > 2$ cases: when $n = 2$ the sectional curvature for $\SL(2)$ with respect to the Hilbert-Schmidt metric is uniformly bounded. 
In Proposition \ref{prop:curvblowupinf} we prove that when $n \geq 3$ there are tangent planes with respect to which the sectional curvature grows unboundedly as we move toward infinity. 
We also derive explicit formulae for the Jacobi equation.

In \cite{AffineFluid3D} a Virial identity was given, and certain conditions on unboundedness of geodesics were given for the motion on $\SL(3)$. These we generalize to $\SL(n)$ and summarize in geometric language in Proposition \ref{prop:poscurv:unbnd} and Theorem \ref{thm:suff:cond:unbnd}.

The main novelty of this section is a new formulation \eqref{eq:def:formulation} of the geodesic equation on $\SL(n)$ as a first order system, in which the conserved angular momentum is explicitly decoupled. This reduces the number of degrees of freedom from $2(n^2 - 1)$ (the dimension of the cotangent bundle of $\SL(n)$) to $n^2 + n - 2$.

We then focus on three special families of geodesics. First, we study linear geodesics in Section \ref{sect:linear}, which are inherited from the ambient space $\Ms(n)$ of square $n \times n$ matrices. We classify these geodesics. All are of the form
\[A(t) = B(I + tM),\]
Where $B \in \SL(n)$, $I$ is the identity matrix, and $M$ is nilpotent, i.e. $M^k = 0$ for some positive integer $k$.
These geodesics were already identified and studied in the 3 dimensional case in \cite{AffineFluid3D}. 

An interesting question to ask is the dynamical stability of this family of linear geodesics. We do not answer the full nonlinear question in this paper. Instead, we study the linear stability of the geodesics by analyzing Jacobi fields along them.  
We show in Theorem \ref{thm:jacobi:rate} that along any linear geodesic the norm of a Jacobi field $J$ grows at most linearly in time and that $|\dot{J}|$ is bounded. For the case $B = I$ and $M^2 = 0$ we fully solve the Jacobi equation in Proposition \ref{prop:jacobisolution}. We find that all Jacobi fields along such geodesics are asymptotically linear.

The second family of geodesics we study are those given in the exponential form
\[A(t) = Be^{tC}\]
for $B \in \SL(n)$ and $C \in \sl(n)$. We note that the exponential here is the Lie group exponential map of $\SL(n)$; as the Hilbert-Schmidt metric is not the bi-invariant pseudo-Riemannian metric on $\SL(n)$, such exponential curves are not automatically geodesic. 
In Section \ref{sect:exp} we explicitly construct all such exponential geodesics, and show they fall into 2 categories. If $C^2 = 0$ then these are linear geodesics as previously discussed. 
When $C^2$ is non-zero, our analysis shows that $C$ is necessarily a non-degenerate skew-symmetric matrix. 
This is only possible when the dimension $n$ is even. 
Physically this corresponds to a stationary solution of the incompressible fluid flow, where the ball is simultaneously spinning about $n/2$ mutually orthogonal planes of rotation, with the speed of rotation (eigenvalues of $C$) carefully balancing the axial lengths of the resulting ellipsoid (eigenvalues of $B$). 
In particular, in contrast to the results in \cite{AffineFluid3D} and \cite{AffineFluid2D}, in all even dimensions strictly greater than 2, there are infinitely many stationary solutions to the affine fluid flow problem. 

Inspired by the spinning ellipsoidal solutions, we study for our third family those solutions corresponding to spinning ellipsoids whose axial lengths and rotational speeds are \emph{not} carefully balanced (indeed, we even allow some rotational speeds to vanish). These we treat in Section \ref{sect:bd}. 
With respect to the formulation \eqref{eq:def:formulation}, 
such solutions are represented in \emph{block-diagonal} form, up to orthogonal transformations; this further reduces the degree of freedom from $n^2 + n - 2$ to $n-2$ (when $n$ is even) or $n-1$ (when $n$ is odd). 

In even dimensions, our Theorem \ref{thm:boundedness} shows that such geodesics with no vanishing rotational speeds must be bounded. We furthermore provide sufficient conditions for the resulting solution to be periodically pulsating. 
In particular, all such block-diagonal, bounded geodesics in dimensions 2 and 4 are periodic. 
The analogous statement is expected to be false in dimensions 6 and higher; we provide numerical simulations hinting at the possibility of chaotic behavior. 

Additionally, we generalize to all dimensions $\geq 3$ the swirling and shear flows analyzed in \cite{AffineFluid3D}. 
Their properties are summarized in Theorem \ref{thm:swfl}; these geodesics are unbounded, and their asymptotic behavior can be precisely described. 
Qualitatively these swirling and shear flows are very different from the linear flows described earlier. 
For the linear flows, our Jacobi field analysis suggests some degree of linear stability under small perturbations. 
As discussed after the proof of Theorem \ref{thm:swfl}, when $n$ is even, the swirling and shear flows are generally \emph{unstable}: while these solutions themselves are unbounded, there exists arbitrarily small initial data perturbations which lead to the perturbed solutions being bounded and periodic.

\section{Differential geometry of \texorpdfstring{$\SL(n)$}{SL(n)}}
\label{sect:dg}
Throughout this paper we will assume $n \geq 2$. The set $\Ms(n)$ of $n\times n$ matrices with real entries can be naturally identified with $\R^{n\times n}$. Under this identification, the Hilbert-Schmidt inner product
\begin{equation}\label{eq:hsip-def}
    \langle A, B\rangle := \tr AB^T
\end{equation}
can be seen as the usual Euclidean inner product on $\R^{n\times n}$. 
Regarding $\SL(n)$ as a submanifold (as the $1$-level set of the defining function $\det: \Ms(n) \to \R$) of $\Ms(n)$, it can be equipped with the Riemannian metric induced by \eqref{eq:hsip-def}. 
For convenience we will refer to this as \emph{the HS metric} on $\SL(n)$.
In this section we present a preliminary discussion of the geometry of $\SL(n)$ with this HS metric. 

\subsection{Useful linear algebra facts}

Prior to entering our discussion of the geometry, it is useful to first record some linear algebraic facts that will be used later. 

\begin{proposition}\label{prop:productprop}
    Suppose $P$ is a symmetric positive definite matrix. Given a matrix $V$, write $V = \tilde{V}+ \hat{V}$ where $\tilde{V}$ is symmetric and $\hat{V}$ antisymmetric. Then
    \begin{enumerate}
        \item $\tr(VPVP) = \tr(\tilde{V}P\tilde{V}P) + \tr(\hat{V}P\hat{V}P)$; 
        \item $\tr(\tilde{V}P\tilde{V}P) > 0$ unless $\tilde{V} = 0$; 
        \item $\tr(\hat{V} P \hat{V}P) < 0$ unless $\hat{V} = 0$.
    \end{enumerate}
\end{proposition}
\begin{proof}
    Since $P$ is symmetric, positive definite, we can write $P = R^2$ where $R$ is also symmetric, positive definite. Note that $\tr(VPVP) = \tr(RVR\cdot RVR)$. Observe that by construction $R\tilde{V}R$ is symmetric and $R\hat{V}R$ is antisymmetric, and hence the two are orthogonal under the Hilbert-Schmidt inner product, and proving the first claim. The second and third claims follow from the fact that the Hilbert-Schmidt inner product is positive definite.
\end{proof}

We record here also the trace inequality of von Neumann: let $A, B$ be arbitrary square matrices, and let $\alpha_i$ and $\beta_i$ be their singular values recorded in descending order, then  
\begin{equation}\label{eq:vntrin}
    |\tr AB| \leq \sum_{i = 1}^n \alpha_i \beta_i.
\end{equation}
See \cite{MVN} for a proof of \eqref{eq:vntrin}. A useful consequence is that for any symmetric, positive semi-definite matrices $P, Q$, we have
\begin{equation}\label{eq:trtrin}
    |\tr PQ| \leq \tr P \cdot \tr Q.
\end{equation}
This inequality is sharp. A consequence of this inequality is the following version of Cauchy-Schwarz inequality for the HS inner product:
\begin{equation}\label{eq:matCS}
    \langle AB, AB\rangle = \tr(ABB^T A^T) = \tr(BB^T A^TA) \leq \tr(BB^T) \tr(A^TA) = \langle B, B\rangle \langle A, A\rangle. 
\end{equation}

In the discussion we will use some facts concerning nilpotent matrices, which are matrices $M\in \Ms(n)$ for which $M^k = 0$ for some positive integer $k$. 
The \emph{nilpotency index} of such a matrix is defined to be the smallest such $k$ (necessarily $\leq n$).
A main result we need is that 
\begin{lemma}\label{lem:nilpotconj}
    A matrix $M\in \Ms(n)$ is nilpotent if and only if there exists an element $O\in\SO(n)$ such that $O M O^{-1}$ is upper triangular with vanishing diagonal. 
\end{lemma}
\begin{proof}
    We provide a proof for lack of a better reference. For the ``if'' direction, it suffices to observe that any strictly upper triangular matrix $T$ satisfies $T^n = 0$. For the ``only if'' direction, first let $k$ be the nilpotency index of $M$. Construct the flag of subspaces $\{0\} \subsetneq \ker M \subsetneq \ker M^2 \subsetneq \cdots \subsetneq \ker M^k = \R^n$. Build inductively a positively-oriented orthonormal basis $\{v_1, \ldots, v_n\}$ of $\R^n$ such that $\{v_1, \ldots, v_{\dim(\ker M^k)}\}$ is an orthonormal basis for $\ker M^k$. In this basis $M$ is strictly upper triangular. 
\end{proof}
\begin{corollary}\label{cor:nilpotchar}
    A matrix $M\in \Ms(n)$ is nilpotent if and only if its characteristic polynomial $\det(sI - M) = s^n$. 
\end{corollary}
\begin{proof}
    The ``if'' direction follows from the Cayley-Hamilton theorem. The ``only if'' direction follows from Lemma \ref{lem:nilpotconj}. 
\end{proof}
\begin{corollary}\label{cor:nilpottr}
    If a matrix $M$ is nilpotent, then $\tr M^j = 0$ for every $j \geq 1$. 
\end{corollary}
\begin{proof}
    By Lemma \ref{lem:nilpotconj} the $O M^j O^{-1}$ must also be strictly upper triangular and has trace 0. But $\tr M^j = \tr M^j O^{-1} O = \tr O M^j O^{-1}$. 
\end{proof}

\subsection{Symmetries}

The special linear group $\SL(n)$ acts on $\Ms(n)$ by matrix multiplication, and preserves the determinant, and so elements of $\SL(n)$ induce diffeomorphisms of $\SL(n)$ to itself. 
These mappings are generally not isometries on the HS metric. Indeed, the HS metric does not agree with the Killing form on $\SL(n)$. 

On the other hand,  the rotation group $\SO(n)$ acts isometrically (with respect to the metric \eqref{eq:hsip-def}) on $\Ms(n)$ (both on the left and on the right), and maps $\SL(n)$ to itself. They therefore induce isometries of $\SL(n)$ with the HS metric. This fact will be frequently used in the subsequent discussions. 

The space $\Ms(n)$ also admits a $\Z_2$ action 
\[ A \mapsto A^T \]
whose fixed-point set is $\Ss(n)$, the space of symmetric $n\times n$ matrices. This $\Z_2$ action clearly preserves $\SL(n)$ and is an isometry with respect to \eqref{eq:hsip-def}, and thus descends to an isometry of $\SL(n)$. An immediate consequence is:
\begin{lemma}\label{lem:symtotgeo}
    The set of symmetric matrices $\Ss(n)\cap \SL(n)$ form a totally geodesic submanifold of $\SL(n)$ under the HS metric.  
\end{lemma}
\begin{remark}
    Symmetric matrices have well-defined signatures (the number of positive, zero, and negative eigenvalues). Along a curve in $\Ss(n)$, the eigenvalues (with multiplicity) are continuous functions. When the curve lies on $\Ss(n) \cap \SL(n)$, the eigenvalues cannot vanish. And hence along a curve in $\Ss(n) \cap \SL(n)$, the signature is constant. Therefore the set $\Ss(n)\cap \SL(n)$ can be divided into connected components each with constant signature. 
\end{remark}
\begin{definition}
    We will denote by $\Ssp(n)$ the component of $\Ss(n)\cap \SL(n)$ corresponding to the positive definite matrices. 
\end{definition}
\begin{remark}
	The set $\Ssp(n)$ is connected: given arbitrary symmetric, positive definite matrices $A, B$, the line segment of convex combinations $[0,1]\ni t\mapsto tA + (1-t)B$ is also a symmetric positive definite matrix. Assuming $\det(A) = \det(B) = 1$, as rescaling preserves the property of being symmetric and positive definite, this allows us to generate a path in $\Ssp(n)$ connecting $A$ to $B$. 
\end{remark}

By the polar decomposition, every element in $\SL(n)$ can be \emph{uniquely} written as $OP$ (or equivalently, $PO$), where $O\in \SO(n)$ and $P\in \Ssp(n)$. 
Since multiplication by $O$ is an isometry, we have that each of these copies of $O\Ssp(n)$, where $O\in \SO(n)$ is a fixed rotation, is totally geodesic as well. These submanifolds foliate $\SL(n)$. 

\begin{remark}
     Since $\Ssp(n)$ is fixed by conjugation with orthogonal matrices, $O\Ssp(n) = O (O^T \Ssp(n) O) = \Ssp(n) O$. Thus the submanifolds $O\Ssp(n)$ and $\Ssp(n)O$ are equal (but parametrized differently by $\Ssp(n)$). Similarly any submanifold of the form $O\Ssp(n) \tilde{O}$ where $O, \tilde{O}\in \SO(n)$ can be identified with the submanifold $O\tilde{O} \Ssp(n)$. 
\end{remark}

\subsection{The tangent bundle}

Being a Lie group, $\SL(n)$ is parallelizable: we can identify $T\SL(n)$ with $\SL(n) \times \sl(n)$ where the Lie algebra
\[ \sl(n) := \{ C \in \Ms(n) \mid \tr C = 0 \}. \]
The ambient space $\Ms(n)$ is a vector space, and the standard coordinate systems allows us to identify the tangent space of $\Ms(n)$ with itself. 
Under the embedding of $\SL(n)$ into $\Ms(n)$, the tangent vector $(A,C) \in \SL(n) \times \sl(n) \cong T\SL(n)$ can be identified with either $AC$ or $CA\in \Ms(n) \cong T_A\Ms(n)$ depending on whether one chooses to work with left- or right-invariant vector fields. 

The symmetries given by the action of the rotation group gives Killing vector fields on $\SL(n)$ with the HS metric. These corresponds to elements of 
\[ \so(n) := \{ Z \in \Ms(n) \mid Z + Z^T = 0 \}.\]
As both the left and right actions by $\SO(n)$ are isometries, this means that in the Euclidean coordinates for $\Ms(n)$, both $AZ$ and $ZA$ are Killing vector fields for any $Z\in \so(n)$.

\begin{proposition}\label{prop:alg-indep}
    The vector fields $\{AZ \mid Z\in\so(n)$\} are linearly independent over $\R$ from the vector fields $\{ZA \mid Z\in \so(n)\}$; therefore the rotation action gives us a total of $n(n-1)$ independent Killing fields. 
\end{proposition}
\begin{proof}
    Suppose not, then there exists non-zero $Z_1, Z_2\in \so(n)$ such that $AZ_1 = Z_2 A$ for all $A\in \SL(n)$. Evaluating at $A = I$ requires $Z_1 = Z_2$. Suppose WLOG the $(1,2)$th entry of $Z_1$ is non-zero and equals $\alpha$. Evaluating at $A = \mathrm{diag}(2, 1/2, 1, 1, \ldots, 1)$ we see that 
    \[ A Z_1 = \begin{pmatrix} 0 & 2\alpha & \cdots \\
    - \frac12 \alpha & 0 \\
    \vdots && \ddots \end{pmatrix} \neq 
    \begin{pmatrix} 0 & \frac12\alpha  & \cdots \\
    - 2 \alpha & 0 \\
    \vdots && \ddots \end{pmatrix} = Z_2 A\]
    which gives a contradiction. 
\end{proof}

\begin{lemma}
    Let $Z\in \so(n)$, the vector field given by $AZ + ZA$ is orthogonal to $\Ssp(n)$. 
\end{lemma}
\begin{proof}
    By definition tangent vectors to $\Ssp(n)$ are symmetric matrices. So it suffices to show that $\langle AZ + ZA,S\rangle = 0$ for any $A\in \Ssp(n), S\in \Ss(n), Z\in \so(n)$. This follows from 
    \[ \tr (AZS^T + ZAS^T) = \tr(ZSA + ZAS) \]
    and noting that $SA + AS$ is a symmetric matrix, while $Z$ is anti-symmetric, so $\tr(Z(SA + AS)) = 0$.
\end{proof}

\begin{remark}
    On the other hand, the vector fields given by $AZ - ZA = AZ + Z^TA$ are tangent to $\Ssp(n)$: they correspond to the infinitesimal generators of the conjugation action of $\SO(n)$ on $\SL(n)$, which leaves $\Ssp(n)$ invariant. 
\end{remark}

\subsection{Extrinsic geometry}

By taking the derivative of $\det$ using Jacobi's formula, one gets a unit normal vector field along $\SL(n)$ given by
\begin{equation}
    N(A) := \frac{A^{-T}}{|A^{-T}|};
\end{equation}
throughout this manuscript we will use $A^{-T}$ to denote the inverse transpose of a square matrix, and \[ 
|B|:= \sqrt{\langle B, B\rangle} = \sqrt{ \tr B^T B} \]
the Hilbert-Schmidt norm of a square matrix. 
\begin{remark}
    By the arithmetic-mean-geometric-mean inequality, the Hilbert-Schmidt norm of an element of $\SL(n)$ is bounded below by $n$. This also means that $|\bullet|$ is a smooth function on $\SL(n)$.
\end{remark}

Now let $X, Y \in T_A\SL(n)$. 
The second fundamental form is given by $\sff(X,Y)|_A := \langle -D_X N, Y \rangle |_A$. Here $D_X N$ is computed by first letting $\tilde{A}$ be a curve in $\SL(n)$ with $\tilde{A}(0) = A$ and $\frac{d}{dt} \tilde{A}(0) = X$, and then computing $\frac{d}{dt} N(\tilde{A}(t))$. We find
\begin{equation}\label{normalderivative}
\begin{aligned}
    D_X N &= \left . \frac{d}{dt}N(\tilde{A}(t)) \right |_{t = 0}\\
    &= - \frac{A^{-T}A'^TA^{-T}}{|A^{-1}|} + \frac{\langle A^{-1}, A^{-1}A'A^{-1}\rangle}{|A^{-1}|^2} \cdot N(A).
\end{aligned}
\end{equation}
by the quotient rule and some simplification. Thus we arrive at
\begin{equation}\label{eq:formsff}
   \sff(X,Y) = \frac{\tr A^{-1}XA^{-1}Y}{|A^{-1}|}
\end{equation}
(noting that $N(A)$ is orthogonal to $Y$).
If we parametrize the tangent bundle $T\SL(n)$ by $\SL(n)\times \sl(n)$ as described in the previous section, then for $X, Y \in \sl(n)$ this yields 
\[
    \sff(AX,AY) = \frac{\tr XY}{|A^{-1}|}.
\]
\begin{lemma}[Second fundamental form bound] \label{lemma:sffbnd}
    Let $X,Y\in T_A\SL(n)$. Then 
    \[ |\sff(X,Y)| \leq |A^{-1}|~ |X|~ |Y|. \]
\end{lemma}
\begin{proof}
    By Cauchy's inequality we can write 
    \[ |\tr(A^{-1} X A^{-1}Y)| = |\langle X^T A^{-T}, A^{-1} Y| \rangle \leq |A^{-1} X| |A^{-1}Y|.\]
    Then, noting that
    \[ |A^{-1}X|^2 = \tr X^T A^{-T} A^{-1} X = \tr A^{-T} A^{-1} X X^T \]
    and the latter is the product of the symmetric, positive semi-definite matrices $A^{-T} A^{-1}$ and $X X^T$, we can apply \eqref{eq:trtrin} to conclude
    \[ |A^{-1}X|^2 \leq |A^{-1}|^2 |X|^2. \]
    Substituting these inequalities into \eqref{eq:formsff} we obtain the Lemma. 
\end{proof}

\begin{proposition}
 The \emph{signature} $(n_0, n_+, n_-)$ of the symmetric bilinear form $\sff$ is $(0, \frac12 n(n+1) - 1 , \frac12 n(n-1))$. 
\end{proposition}
\begin{proof}
 Let $W\in T_A\SL(n)$ which we regard as a subspace of $\Ms(n)$. By the polar decomposition we can write $A = OP$ where $P$ is symmetric positive definite. Then \eqref{eq:formsff} gives
 \[ \sff(W,W) = \frac{ \tr P^{-1} O^{-1} W P^{-1}O^{-1}W}{|A^{-1}|} \]
 and our proposition follows from Proposition \ref{prop:productprop}.
\end{proof}

\subsection{The geodesic equation and connection to fluid balls}

A curve $A(t)$ on the submanifold $\SL(n)$ is a geodesic of $\SL(n)$ with respect to the induced HS metric, if and only if, when regarded as a curve in the Euclidean space $\Ms(n)$ its acceleration $A''(t)$ is everywhere parallel to $N(A(t))$. 
This can be expressed in terms of the second fundamental form
\[
    A'' = \sff(A',A')N(A),
\]
or explicitly
\begin{equation}\label{geodesiceq}
    A'' = \frac{\tr A'A^{-1}A'A^{-1}}{\tr A^{-1}A^{-T}}A^{-T}.
\end{equation}

In \cite{AffineFluid3D} it was shown that the solutions to \eqref{geodesiceq} can be associated to solutions of the incompressible Euler equation; here we give a brief summary of the connection.

Let $\Omega\subset \R^n$ be some reference domain. Taking the Lagrangian coordinate point of view, we shall regard the fluid body at time $t$ to be described by a mapping $\phi_t:\Omega \to \R^n$ that is a diffeomorphism onto its image, and require that the trajectories $t \mapsto \phi_t(x)$ for a fixed $x\in\Omega$ to be the fluid flow lines. 
\emph{Incompressibility} of the fluid requires that $\phi_t$ are \emph{volume-preserving diffeomorphisms} (in other words have unit Jacobian determinant). 

It is well-known, at least when the spatial geometry of the domain is fixed (so that $\phi_t(\Omega) = \Omega$; this is the case of a perfect fluid in a domain), that through the Arnold-Euler formulation \cite{ArnoldEuler}, the equations of incompressible fluid flow can be recast as that of geodesic motion relative to some Riemannian metric on the infinite-dimensional Lie group of volume-preserving diffeomorphisms of $\Omega$ to itself. 

We will however consider the situation where the fluid domain has a free boundary, and $\phi_t(\Omega)$ is time-dependent: this represents the (free) evolution of a fluid body surrounded by vacuum. 
In our simplified setting, we shall first impose \emph{first as a constraint} the requirement that the $\phi_t$ are volume-preserving affine mappings; that is to say, there exist functions $o:\R\to\R^n$ and $A:\R\to \SL(n)$ such that $\phi_t(x) = o(t) + A(t) \cdot x$. We shall see later that in certain cases this constraint is \emph{self-consistent}. 

Regarding $\Omega$ as parametrizing a collection of particles with uniform density 1, we can approach the dynamics using a Lagrangian formulation. 
By translating $\Omega$ we can assume without loss of generality that the center of mass of $\Omega$ is at the origin, that is
\begin{equation}\label{eq:com:omega}
    \int_{\Omega} x ~dx = 0.
\end{equation}
The action of a path is given entirely by the total kinetic energy, which can be written as
\[
    S = \int_{t_1}^{t_2} \int_\Omega \frac12 (o'(t) + A'(t) \cdot x)\cdot(o'(t) + A'(t) \cdot x) ~dx ~dt. 
\]
The center-of-mass condition \eqref{eq:com:omega} allows us to simplify (here $|\Omega|$ means the Lebesgue measure of $\Omega$)
\begin{equation}\label{eq:action:omega}
    S = |\Omega| \int_{t_1}^{t_2} \frac12 |o'(t)|^2 ~dt + \int_{t_1}^{t_2} \int_\Omega \frac12 \tr A'(t) xx^T (A'(t))^T ~dx ~dt.
\end{equation}
Denote by the matrix 
\begin{equation}\label{eq:action:momentdef}
    I_\Omega = \int_\Omega xx^T ~dx 
\end{equation}
the \emph{moment of inertia} of the reference body $\Omega$. By construction $I_\Omega$ is symmetric, positive definite, and therefore induces a \emph{inner product} on $\Ms(n)$
\[ \langle M_1, M_2\rangle_\Omega = \tr M_1 I_\Omega M_2^T.\]
One sees then the action \eqref{eq:action:omega} is nothing more than the \emph{energy} of a path in the Riemannian product manifold $\R^n\times \SL(n)$, where the $\R^n$ factor is equipped with $|\Omega|$ times the standard metric, and the $\SL(n)$ factor is equipped with the induced metric from the $\langle \bullet, \bullet \rangle_\Omega$ inner product on $\Ms(n)$. 

We conclude that the \emph{affinely-constrained} fluid motion is identical to geodesic motion on this Riemannian product manifold. 
In particular, the center of mass motion (described by the $o(t)$ factor) decouples from the $\SL(n)$ motion and travels in straight lines. 
The metric induced by the moment of inertia $I_\Omega$ is identical (up to a scaling factor) the the Hilbert-Schmidt metric whenever the domain $\Omega$ has the same symmetries as the hypercube $[1,1]^n\subset\R^n$. Let us for simplicity assume that this is the case. 

Since this motion is \emph{constrained} by assumption to be affine, it will in general not agree with corresponding Euler flow: the set of affine diffeomorphisms are in general not totally geodesic within the set of all volume-preserving diffeomorphisms (with the appropriate Riemannian metrics). 
Let as assume without loss of generality that $o(t) \equiv 0$ (since the center-of-mass motion decouples). Then the fluid velocity field, in the Eulerian coordinate system, is 
\[ \phi_t(\Omega) \ni y \mapsto   v(t,y) =  A'(t)A(t)^{-1}y.\]
The Euler equation for incompressible flow 
\begin{equation} \label{eq:euler:basic}
\underbrace{\partial_t v + v\cdot \nabla v}_{A''(t) A(t)^{-1}y} = - \nabla p
\end{equation}
is satisfied if $\nabla p(y) = -\sff(A'(t),A'(t)) N(A(t)) A^{-1}(t) y$. This requires the pressure 
\begin{equation}\label{eq:euler:pressure}
    p(y) = -\frac1{2|A^{-1}(t)|} \sff(A'(t),A'(t)) |A^{-1}(t) y|^2.
\end{equation}
Along the boundary $\partial \phi_t\Omega$, however, the free boundary condition requires physically the pressure to be constant. 
This is compatible with \eqref{eq:euler:pressure} if and only if $\Omega$ is a round ball. In particular, if one studies the incompressible Euler flow with initial domain having the geometry of an affine deformation of the ball, and with initial (divergence-free) velocity field being a linear function of the position, then the Euler flow will coincide with the geodesic flow on $\SL(n)$ with the HS metric. 

\begin{remark}\label{rmk:taylorstab}
    The Euler equations for a fluid body surrounded by vacuum comprise of the Euler equation \eqref{eq:euler:basic}, the constant pressure condition, and one additional inequality: the \emph{Taylor sign condition}. This last condition pertains to the stability of the Euler equations as a partial differential equation, and has no effect on the geodesic flow on $\SL(n)$. As discussed in \cite{AffineFluid3D}, the Taylor sign condition is equivalent to $\sff(A',A') \geq 0$ by virtue of \eqref{eq:euler:pressure}; we can thus regard the positivity of the second fundamental form along the geodesic a sort of ``physicality'' property of the corresponding fluid flow. 
\end{remark}

\subsection{Geodesic properties of \texorpdfstring{$\SL(n)$}{SL(n)}}

One easily checks that $\langle A', A'\rangle' = 2 \langle A', A''\rangle = 0$ for any geodesic (using that $A'' \perp T\SL(n) \ni A'$), reflecting the well-known fact that the kinetic energy is conserved along geodesics. 

\begin{lemma}\label{lem:lengthcomp}
    Let $A, B\in \SL(n)$, define the geodesic distance $d(A,B)$ as
    \[ d(A,B) := \inf_{\substack{\gamma: [0,1]\to \SL(n)\\ \gamma(0) = A,  \gamma(1) = B}} \int_0^1 \langle \gamma'(s), \gamma'(s)\rangle^\frac12 ~ ds. \]
    Then $d(A,B) \geq |A - B|$.
\end{lemma}
\begin{proof}
    The length $|A-B|$ is equal to the length of the straight-line segment in $\Ms(n)$ that connects $A$ to $B$, and is characterized by 
    \[ |A-B| = \inf_{\substack{\gamma: [0,1]\to \Ms(n)\\ \gamma(0) = A,  \gamma(1) = B}} \int_0^1 \langle \gamma'(s), \gamma'(s)\rangle^\frac12 ~ ds. \]
    As $\SL(n) \subset \Ms(n)$ necessarily $|A-B| \leq d(A,B)$.
\end{proof}
\begin{corollary}
    Let $A(t)$ be a geodesic in $\SL(n)$. Then 
    \[ |A(t)| \leq |A(0)| + t |A'(0)|. \]
\end{corollary}

\begin{rem}
    In the context of the corresponding fluid flow, the above corollary says that the diameter of the the fluid ball cannot grow faster than linearly in time. (Note that by virtue of incompressibility, growth of diameter means the collapse of some other principal axis/axes.)
\end{rem}

\begin{theorem}
    $\SL(n)$ with the HS metric is geodesically complete.
\end{theorem}
\begin{proof}
    By the Hopf-Rinow theorem \cite{ONeillSemi}*{Chapter 5, Theorem 21}, it suffices to show that closed and bounded (relative to the $d(A,B)$ distance) subsets of $\SL(n)$ are compact. Let $K \subset \SL(n)$ be a closed bounded subset, and therefore there exists some $A\in K$ and some $R > 0$ such that $K \subset \{ B \in \SL(n) \mid d(A,B) \leq R\}$. By Lemma \ref{lem:lengthcomp} we have $K \subset \{B\in \Ms(n) \mid |A-B| \leq R\}$ and hence is bounded in $\Ms(n)$. As $\det$ is a continuous function, we have that its level set $\SL(n)$ is a closed subset of $\Ms(n)$, and hence $K$ is also closed in $\Ms(n)$. Therefore $K$ is closed and bounded in $\Ms(n)$ and by Heine-Borel is compact. 
\end{proof}

Letting $Z\in \so(n)$, we can consider the inner product $\langle A'(t), ZA(t)\rangle$ where $A$ is a geodesic. Taking the time derivative of this quantity we find 
\[
    \langle A', ZA\rangle' = \langle A'', ZA \rangle + \langle A', ZA'\rangle. 
\]
The first factor vanishes since $A''$ is a normal vector and $ZA$ tangential. The second factor vanishes since it equals $\tr A'(A')^T Z^T$ and $A'(A')^T$ is a symmetric matrix, so its trace against $Z$ vanishes. One similarly finds that the quantities $\langle A'(t), A(t) Z\rangle$ to be constants. We have thus recovered the well-known fact that the inner product of the velocity vector of a geodesic and a Killing vector field is a constant of motion for the geodesic. 
 
In this formulation, the element $Z\in \so(n)$ is time-independent. This allows us to reformulate the corresponding conservation laws as the statements
\begin{equation}\label{eq:consAM}
    \frac{d}{dt} \left[(A^T) A' - (A^T)'A\right] = 0 = \frac{d}{dt}\left[ A' A^T - A(A^T)'\right]
\end{equation}
whenever $A:\R\to\SL(n)$ is geodesic. 

By Proposition \ref{prop:alg-indep}, this means that the geodesic flow on the $n^2 - 1$ dimensional manifold $\SL(n)$ has $n(n-1)$ constants of motion that correspond to the rotation symmetry. Together with the Hamiltonian (conserved energy) this gives a total of $n^2 - n + 1$ conserved quantities. And therefore we obtain\footnote{The $\SL(2)$ geodesic motion has been analyzed exhaustively in \cite{AffineFluid2D}.}
\begin{proposition}\label{prop:2disint}
    The geodesic motion on $\SL(2)$ with the HS metric is integrable. 
\end{proposition}

\begin{remark}[Conservation laws] \label{rmk:disc:conslaw}
    The conservation of $\langle A', A'\rangle$, the matrix $A^T A' - (A^T)' A$, and the matrix $A' A^T - A(A^T)'$ have physical interpretations. For a fluid ball, the moment of inertia $I_\Omega$ of \eqref{eq:action:momentdef} is by construction proportional to the identity matrix, and hence the \emph{total kinetic energy} (see \eqref{eq:action:omega}) of the fluid ball (assuming vanishing linear momentum) is, up to a constant of proportionality, the kinetic energy of the geodesic $\langle A',A'\rangle$. 
    
    The \emph{total angular momentum} of the fluid ball is given by the bivector
    \[ \int_\Omega (A(t)x) \wedge (A'(t)x) ~dx. \]
    As a matrix relative to the rectangular coordinates this can be seen to equal
    \[ \int_\Omega A(t) x x^T A'(t)^T - A'(t)x x^T A(t)^T ~dx \]
    and therefore is proportional to the conserved matrix 
    \[ A' A^T - A (A')^T.\]
    
    The matrix $A^T A' - (A^T)' A$ on the other hand is related to the \emph{vorticity} of the affine fluid ball. Recall that for a fluid velocity distribution $v$ (as in \eqref{eq:euler:basic}), the corresponding vorticity is $dv$, the antisymmetric part of its gradient.\footnote{In three dimensions this is usually written in terms of a pseudovector, which is the Hodge dual of $dv$, or equivalently the curl of $v$.} 
    The vorticity is, generally speaking, \emph{not} a conserved quantity for incompressible flow. In the affine setting, however, we can compute the vorticity to be (the spatially constant two-form)
    \[ dv(t,y) = A'(t)A(t)^{-1} - A(t)^{-T} (A'(t))^T.\]
    The conserved matrix $A^T A' - (A^T)' A$ can thus be interpreted as the \emph{pull-back} of the vorticity two form $dv$ from the Eulerian coordinates to the Lagrangian coordinates $(t,x) \in \R\times\Omega$. 
\end{remark}

The conservation law \eqref{eq:consAM} inspires the following first-order formulation of the geodesic equation \eqref{geodesiceq}.\footnote{One can alternatively swap $A^T$ and $A$ and define $\beta = A^{-T} A^{-1}$ etc., in which case the conserved $\zeta$ changes from the pull-back vorticity to the total angular momentum.}
Define
\begin{equation}\label{eq:def:formulation}
\begin{aligned}
    \BB&:=A^{-1}A^{-T},\\
    \WW&:=A^TA'+(A^T)'A,\\
    \ZZ&:=A^TA' - (A^T)'A.
\end{aligned}
\end{equation}
We note that $\BB \in \Ssp(n)$, $\WW$ is symmetric, and $\ZZ$ is antisymmetric. The compatibility condition that $A' \in T_A\SL(n)$ is the requirement that $\tr(A^{-1} A') = 0$; this becomes 
\begin{equation}\label{eq:fo:compat}
    0= \tr(\BB(\WW + \ZZ)) = \tr(\BB\WW).
\end{equation}
In the second inequality we used that $\BB$ is symmetric and $\ZZ$ is antisymmetric and so they are orthogonal with respect to the HS inner product. 
Then
\begin{equation}\label{eq:sffdecompbeta}
\sff(A',A') =\frac{\tr((\WW+\ZZ)\BB(\WW+\ZZ)\BB)}{4 \sqrt{\tr \BB}}.
\end{equation}
We note that $\tr(\WW\BB\ZZ\BB) = \tr(\BB\WW\BB\ZZ)=0$ by Proposition \ref{prop:productprop}.
So whenever $A(t)$ is a geodesic, \eqref{geodesiceq} implies the following system of equations:
\begin{subequations}
\begin{align}
    \BB'&=-\BB\WW\BB \label{eq:fo:bb}\\
    \ZZ'&=0 \\
    \WW'&=\frac12 (\WW+\ZZ)^T \BB (\WW+\ZZ) + \frac12 \frac{\tr \WW \BB \WW \BB}{\tr \BB} I + \frac12 \frac{\tr \ZZ \BB \ZZ \BB}{\tr \BB} I \label{eq:fo:ww}
\end{align}
\end{subequations}

The equations \eqref{eq:fo:bb}--\eqref{eq:fo:ww} have a natural interpretation in terms of the polar decomposition. Let $A(t) = O(t) P(t)$ via the polar decomposition, which requires that $A^T A = P^2$.
The quantity $\BB = (A^T A)^{-1} = P^{-2}$ is essentially the ``radial part'' of the polar decomposition; with $\WW$ essentially its time derivative.  
Since $\ZZ$ is constant in time, we see that the equations \eqref{eq:fo:bb} and \eqref{eq:fo:ww} give a closed system for the dynamics in the ``radial direction'' of the polar decomposition. 
That such a reduction is available is not too surprising: morally speaking one may expect that the conservation of angular momenta to relate the ``angular'' motion (motion of the $O(t)$ factor in the polar decomposition) to ``radial'' motion; this is similar to the formulation of the equation of motion of a particle in a spherically symmetric potential well in terms of radial motion against an ``effective potential'' that incorporates the effect of the conserved angular momentum.

\begin{remark}[Canonical forms] \label{rmk:canonical}
    When performing \emph{pointwise} computations it is often convenient to take advantage of the fact that left and right multiplication by elements of $\SO(n)$ are isometries. Given an element $U\in \SO(n)$, we will denote $A^U = U A U^{-1}$ the conjugate action. If $A(t)$ is a geodesic, so is $\tilde{A}(t) = A(t)^U$ for any fixed $U$. We note that the corresponding $\tilde{\BB},\tilde{\WW}, \tilde{\ZZ}$ satisfy
    \[ \tilde{\BB} = \BB^U, \tilde{\WW} = \WW^U, \tilde{\ZZ} = \ZZ^U.\]
    And therefore, we may often assume, at one fixed point in time, without loss of generality, that either $\BB$ or $\WW$ are purely diagonal, or any other canonical forms that can be achieved through conjugation by $\SO(n)$ matrices. 
\end{remark}

\begin{remark}[Block-diagonal preservation] \label{rmk:block-diag}
	It should be noted that if $\BB, \ZZ, \WW$ have the same block-diagonal form, then the block diagonal form is preserved, which can be seen from the form of the differential equations. Note however that due to the trace terms in \eqref{eq:fo:ww}, the evolution of the separate blocks do \emph{not} decouple. This reflects the fact that the geometry of $\SL(n)$, even restricted to just the cases where $\BB, \ZZ, \WW$ are block diagonal, does not arise as a Cartesian product manifold.
\end{remark}

\subsection{Riemann curvature and the Jacobi Equation}
We begin with computing the Riemann curvature. Recall from the Gauss Equation \cite{ONeillSemi}*{Chapter 4, Theorem 5}, if $X, Y, Z\in T_A\SL(n)$, 
\begin{equation}\label{eq:def:riem}
    R(X,Y)Z =  \sff(Y,Z) D_X N - \sff(X,Z) D_Y N.
\end{equation}
By \eqref{normalderivative} and \eqref{eq:formsff}, we have, for $X,Y,Z\in \sl(n)$
\begin{multline}\label{eq:riemform}
    R(XA,YA)ZA = -\frac{\tr YZ}{|A^{-1}|} \left( \frac{X^T A^{-T}}{|A^{-1}|} - \frac{\tr  A^{-1} X^T A^{-T}}{|A^{-1}|^2} N(A)\right) \\
    + \frac{\tr XZ}{|A^{-1}|} \left( \frac{Y^T A^{-T}}{|A^{-1}|} 
    - \frac{\tr A^{-1}  Y^T A^{-T}}{|A^{-1}|^2} N(A)\right)
\end{multline}
We take the inner product with $WA$ with $W \in \sl(n)$ to get the Riemann tensor
\begin{equation}\label{eq:riembilin}
    \langle R(XA,YA)ZA, WA\rangle = \frac{ \tr(XZ)\tr(YW) - \tr(YZ)\tr(XW)}{|A^{-1}|^2}.
\end{equation}

We can also evaluate the sectional curvature: 
for $X, Y\in T_A\SL(n)$ that are not parallel, 
\begin{equation}\label{eq:sectcurv}
\begin{aligned}
    K(X, Y;A) &= \frac{\langle R(X,Y)X, Y\rangle}{|X|^2|Y|^2 - \langle X, Y\rangle^2}\\
    &= \frac{1}{|A^{-1}|^2}\frac{ \tr(A^{-1}YA^{-1}Y)\tr(
    A^{-1}XA^{-1}X)
    - [\tr(A^{-1}XA^{-1}Y)]^2}{|X|^2|Y|^2-\langle X, Y\rangle^2}.
\end{aligned}
\end{equation}
Notice that as we move toward infinity in $\SL(n)$ we have that $|A^{-1}| \to \infty$ and $|A| \to \infty$ (though at potentially vastly different rates). 
Applying Lemma \ref{lemma:sffbnd} to \eqref{eq:sectcurv}, with $X, Y$ assumed orthogonal, we find the following sectional curvature bound
\begin{equation}
    |K(X,Y;A)| \leq |A^{-1}|^2. 
\end{equation}
This bound turns out to be \emph{sharp} when $n \geq 3$, as we see below. 
\begin{proposition}\label{prop:curvblowupinf}
  When $n \geq 3$, there exists a sequence of locations $A_i$ and a sequence of nonparallel tangent vectors $X_i, Y_i\in T_{A_i}\SL(n)$ such that 
  \[ \liminf_{i \to \infty}  \frac{|K(X_i, Y_i; A_i)|}{|A^{-1}_i|^2} > 0.\]
\end{proposition}
\begin{proof}
    Let $A_\lambda = \mathrm{diag}(\lambda^{-1},\lambda^{-1}, \lambda^2,1, \ldots, 1)$. This means that $|A_{\lambda}^{-1}|^2 = 2 \lambda^2 + \lambda^{-4} + (n-3)$. 
    Fix $X = \mathrm{diag}(1, -1,0, \ldots, 0)$ and  
    \[ Y = \begin{pmatrix} & 1 \\ -1 \\ & & 0 \\
    & & & \ddots \end{pmatrix} \]
    Then 
    \[ K(XA_\lambda, Y A_\lambda; A_\lambda) = \frac{1}{|A_\lambda^{-1}|^2} \frac{ \tr Y^2 \tr X^2 - \tr XY }{|XA_{\lambda}|^2 |YA_{\lambda}|^2 - \langle XA_{\lambda}, Y{A\lambda}\rangle}. \]
    Now, $\tr Y^2 = -2$, $\tr X^2 = 2$, and $\tr XY = 0$. And finally 
    \[|X A_\lambda|^2 = 2\lambda^{-2} = |Y A_\lambda|^2 \]
    while $\langle XA_\lambda, YA_\lambda\rangle = 0$. So we've found
    \[ K(XA_\lambda, Y A_\lambda; A_\lambda) = - \frac{\lambda^4}{2 \lambda^2 + \lambda^{-4} + n-3} \]
    and we have 
    \[ \lim_{\lambda \to +\infty} \frac{ K(XA_{\lambda}, YA_{\lambda}; A_{\lambda})}{|A_{\lambda}^{-1}|^2} = - \frac14 .\qedhere\]
\end{proof}

\begin{rem}
    Proposition \ref{prop:curvblowupinf} does not hold in dimension 2. Using the rotational symmetry, we can assume without loss of generality that $A$ is diagonal, which requires $A = A_{\lambda} =  \mathrm{diag}(\lambda, \lambda^{-1})$. Its tangent space is spanned by $X_{\lambda}A_\lambda, Y_{\lambda}A_\lambda, Z_{\lambda}A_\lambda$ where
    \[ X_\lambda = \begin{pmatrix}0 & 0 \\ \lambda^{-1} & 0 \end{pmatrix}, \quad Y_{\lambda} = \begin{pmatrix} 0 & \lambda \\ 0 & 0 \end{pmatrix}, \quad Z_{\lambda} = \frac{1}{\sqrt{\lambda^2 + \lambda^{-2}}} \begin{pmatrix} 1 \\ & -1\end{pmatrix}.\]
    It is easily verified that $X_\lambda, Y_\lambda, Z_\lambda \in \sl(2)$, and that $\{X_\lambda A_\lambda, Y_\lambda A_\lambda, Z_\lambda A_\lambda\}$ form an orthonormal basis of $T_{A_\lambda} \SL(2)$. Relative to this basis we can compute
    \begin{gather}\label{eq:2dsffcomp}
        \sff(X_\lambda A_\lambda, X_\lambda A_\lambda) = \sff(Y_\lambda A_\lambda, Y_\lambda A_\lambda ) = 0 ;\\
        \sff(X_\lambda A_\lambda, Z_\lambda A_\lambda) = \sff(Y_\lambda A_\lambda, Z_\lambda A_\lambda ) = 0;\\
        \sff(X_\lambda A_\lambda, Y_\lambda A_\lambda) = 1 ;\\
        \sff(Z_\lambda, A_\lambda, Z_\lambda A_\lambda) = \frac{2}{\lambda^2 + \lambda^{-2}} \leq 1.
    \end{gather}
    This shows that the second fundamental form is uniformly bounded, in stark contrast to the estimate in Lemma \ref{lemma:sffbnd}.
    This further implies that the sectional curvature is uniformly bounded over $\SL(2)$, giving another significant difference between the $n = 2$ and $n \geq 3$ cases. 
\end{rem}

As it turns out, the sectional curvature is highly anisotropic on $\SL(n)$; this is already seen to some degree in the computation \eqref{eq:2dsffcomp}. 
In particular, in contrast to Proposition \ref{prop:curvblowupinf}, we will prove below Corollary \ref{cor:lingeo:curvdecay} and Theorem \ref{thm:jacobi:rate} which show, and take advantage of, the fact that along certain geodesics, the sectional curvature decays like the inverse square of the affine parameter. 

We now compute the Jacobi equation. Recall that the Jacobi equation for a Jacobi field $J$ along a geodesic $A$ is given by
\begin{equation}\label{eq:jacobibasic}
    \frac{D^2J}{dt^2} = R(J, A') A'.
\end{equation}
If $f(s, t)$ is a two-parameter family of geodesics satisfying $f(0, t) = A(t)$, then the Jacobi equation is solved by
\[J = \frac{\partial f}{\partial s}(0, t),\]
measuring the rate of divergence of geodesics around $A$. 

The symbol $\frac{D}{dt}$ in \eqref{eq:jacobibasic} is the covariant derivative along the geodesic $A$. 
When working with $A(t)$ as elements in the ambient $\Ms(n)$, it is more convenient by directly taking the variation of \eqref{geodesiceq}. This yields
\begin{multline}\label{eq:jacobicoord}
    J'' = \frac{2 \tr (J' A^{-1} A' A^{-1}) - 2 \tr(A' A^{-1} J A^{-1} A' A^{-1})}{\tr A^{-1}A^{-T}} A^{-T} \\
    + \frac{\tr(A'A^{-1}A'A^{-1})}{(\tr A^{-1} A^{-T})^2} ( 2\tr(A^{-1}JA^{-1}A^{-T})) A^{-T} \\
    - \frac{\tr(A'A^{-1}A'A^{-1})}{\tr A^{-1} A^{-T}} A^{-T} J^T A^{-T}.
\end{multline}
It is worth noting that since $J$ is tangent to $\SL(n)$, we have $\tr J A^{-1} = 0$. This implies
\begin{equation}
    \tr J' A^{-1} = \tr J A^{-1} A' A^{-1} 
\end{equation}
and
\begin{multline*}
    \tr J'' A^{-1} = (\tr J' A^{-1})' + \tr J' A^{-1} A' A^{-1} = (\tr J A^{-1} A' A^{-1})' + \tr J' A^{-1} A' A^{-1} \\
    = 2 \tr J' A^{-1} A' A^{-1} - 2 \tr J A^{-1} A' A^{-1} A' A^{-1} + \tr J A^{-1} A'' A^{-1}.
\end{multline*}
Using \eqref{geodesiceq} we finally get
\begin{multline}
    \tr J'' A^{-1} = 2 \tr J' A^{-1} A' A^{-1} - 2 \tr J A^{-1} A' A^{-1} A' A^{-1}\\ + \frac{1}{|A^{-1}|^2} \tr (J A^{-1} A^{-T} A^{-1}) \cdot \tr (A' A^{-1} A' A^{-1}).
\end{multline}
So we can alternatively also write \eqref{eq:jacobicoord} as 
\begin{equation}\label{eq:jacobieq}
    J'' - \langle J'', N\rangle N= 
    \frac{\sff(A',A')}{\tr A^{-1} A^{-T}} ( \tr(A^{-1}JA^{-1}A^{-T})) N 
    - \frac{\sff(A',A')}{|A^{-1}|} A^{-T} J^T A^{-T}.
\end{equation}
The equation \eqref{eq:jacobieq} can be interpreted as requiring that the vector field 
\[ J'' + \frac{\sff(A',A')}{|A^{-1}|} A^{-T} J^T A^{-T} \]
be purely normal to $\SL(n)$. 

\subsection{The Virial identity and unboundedness}
Observe that 
\begin{equation}
    \frac{d}{dt} |A|^2 = \tr\WW
\end{equation}
where $\WW$ is defined in \eqref{eq:def:formulation}. By \eqref{eq:fo:ww} this means the \emph{Virial identity} 
\begin{equation}\label{eq:virial} 
     \frac{d^2}{dt^2} |A| =  \frac{|A'|^2}{|A|} + \frac{n}{|A|^2} \sff(A', A')
\end{equation}
holds. A particular consequence of the Virial identity is that 
\begin{proposition}[See \cite{AffineFluid3D}] \label{prop:poscurv:unbnd}
 If $A(t)$ is a geodesic on $\SL(n)$, such that $\sff(A',A')$ is positive for all $t > t_0$, then $|A|$ grows linearly as $t\nearrow \infty$.  
\end{proposition}

As indicated in Remark \ref{rmk:taylorstab}, this implies that \emph{all physically meaningful affine solutions} to the incompressible Euler equations are unbounded on $\SL(n)$. 
The positivity of $\sff(A',A')$ can be guaranteed for the duration of the flow for certain special initial data.  

\begin{theorem}\label{thm:suff:cond:unbnd}
    Let $A_0 \in \SL(n)$ and $A_1 \in T_{A_0}\SL(n) \setminus \{0\}$ be such that one of the following conditions hold:
    \begin{itemize}
        \item relative to the polar decomposition of $A_0 = OP$, the matrix $O^{T}A_1$ is symmetric;
        \item the matrix $A^T_0 A_1$ is symmetric;
        \item the matrix $A_1 A^T_0$ is symmetric.
    \end{itemize}
    Then the geodesic $A(t)$ on $\SL(n)$ with initial data $A(0) = A_0$ and $A'(0) = A_1$ satisfies $\sff(A'(t),A'(t)) > 0$ for all $t$, and hence is unbounded as $t \to \pm\infty$.  
\end{theorem}

\begin{remark}
    The first of the conditions corresponds to the geodesic being contained in one of the totally geodesic leaves $O\Ssp(n)$. By Remark \ref{rmk:disc:conslaw} the second of the conditions corresponds to vanishing vorticity, and the third of the conditions corresponds to to vanishing total angular momentum. 
\end{remark}

\begin{proof}
    In the first case, observe that since $A'(0) \in T_{A_0}(O\Ssp(n))$, by the total geodesy (Lemma \ref{lem:symtotgeo}) of symmetric matrices, we have that $O^TA'(t)$ is symmetric for all $t$. Therefore $\sff(A', A') = \frac{1}{|A^{-1}|} \tr P^{-1} (O^T A') P^{-1} (O^T A') > 0$ by Proposition \ref{prop:productprop}. 
    
    In the second case, using our decomposition \eqref{eq:def:formulation}, we see that the conserved $\ZZ \equiv 0$. And hence by \eqref{eq:sffdecompbeta} we have $\sff(A',A') = \frac{1}{4\sqrt{\tr\BB}} \tr \WW\BB\WW\BB > 0$ again using Proposition \ref{prop:productprop}.
    
    The proof of the third case is almost identical to the second, using now that the antisymmetric part of  $A'(t) A^T(t)$ is conserved and hence $A'(t)A^T(t)$ is symmetric for all time. 
\end{proof}

In Proposition \ref{prop:poscurv:unbnd}, the positive kinetic energy in \eqref{eq:virial} is not used. In dimension two it turns out incorporating this factor we can strengthen Proposition \ref{prop:poscurv:unbnd}, without explicitly solving \eqref{geodesiceq} as in \cite{AffineFluid2D}. (Recall from Proposition \ref{prop:2disint} that in dimension two \eqref{geodesiceq} is completely integrable.) 

\begin{theorem}\label{thm:2dbndedorcstnorm}
    Let $A(t)$ be a geodesic on $\SL(2)$ with the HS metric. Then either $A(t)$ is unbounded, or $|A(t)| \equiv \sqrt{2}$. 
\end{theorem}
\begin{proof}
    We first prove using \eqref{eq:virial} that $|A(t)|$ is convex. It turns out to be convenient to use the formulation \eqref{eq:def:formulation}; specifically it suffices to show that the right-hand-side of \eqref{eq:fo:ww} has non-negative trace for any possible values of $\beta, \zeta, \omega$ that correspond to $A, A'$ of a curve in $\SL(2)$.  
    
    To simplify the computation, we make use of Remark \ref{rmk:canonical} and assume without loss of generality that $\BB$ is diagonalized. Since $\BB\in \Ssp(n)$ and $\ZZ$ is antisymmetric we have that the canonical form gives
    \[ \BB = \begin{pmatrix} b & 0 \\ 0 & b^{-1} \end{pmatrix}, \qquad \ZZ = \begin{pmatrix} 0 & z \\ -z & 0 \end{pmatrix}. \]
    By the compatibility condition \eqref{eq:fo:compat}, we must have 
    \[ \WW = \begin{pmatrix} cb^{-1} & w \\ w & -cb \end{pmatrix}.\]
    Based on this ansatz, we with to compute and show the non-negativity of 
    \[ \frac12 \tr(\WW - \ZZ)\BB(\WW +\ZZ) + \frac{\tr \WW\BB\WW\BB}{\tr\BB} + \frac{\tr \ZZ\BB\ZZ\BB}{\tr\BB}. \]
    We compute term by term
    \begin{align*}
        \frac12 \tr(\WW -\ZZ) \BB(\WW + \ZZ) & = \frac12 \left( c^2 (b + b^{-1}) + b(w+z)^2 + b^{-1}(w-z)^2 \right) \\
        \frac{\tr \WW\BB\WW\BB}{\tr \BB} + \frac{\tr \ZZ \BB\ZZ\BB}{\tr\BB}& = \frac{2}{b + b^{-1}} \left( c^2 + w^2 - z^2 \right)
    \end{align*}
    So we find
    \begin{equation}
        \tr\WW' = c^2 \left( \frac{b+b^{-1}}{2} + \frac{2}{b + b^{-1}} \right) + \frac{b}2(w+z)^2 + \frac1{2b}(w-z)^2 + \frac{2(w^2 - z^2)}{b+b^{-1}}.
    \end{equation}
    Noting that by the arithmetic-mean-geometric-mean inequality, $\frac12 (b + b^{-1}) \geq \sqrt{b b^{-1}} = 1$, we have 
    \begin{align*}
        \frac{b}2(w+z)^2 + \frac1{2b}(w-z)^2 &+ \frac{2(w^2 - z^2)}{b+b^{-1}} \\
        & = \left(1 - \frac{2}{b+b^{-1}}\right) \left( 
        \frac{b}2(w+z)^2 + \frac1{2b}(w-z)^2\right) \\
        & \quad + \frac{1}{b+b^{-1}} \left( \sqrt{b}(w+z) + \sqrt{b^{-1}}(w-z)\right)^2.
    \end{align*}
    Therefore we have shown that $\tr \WW' \geq 0$ and hence $|A(t)|$ is convex. This implies that $A(t)$ is either unbounded, or $|A(t)|$ is constant and $\tr \WW' \equiv 0$. 
    
    In the latter case, as $\tr\WW'$ is a sum of non-negative quantities, each must individually equal zero. This requires therefore 
    \begin{align*}
        c^2 \left( \frac{b+b^{-1}}{2} + \frac{2}{b + b^{-1}} \right) = 0 &\implies c^2 = 0,\\
        \left(1 - \frac{2}{b+b^{-1}}\right) \left( 
        \frac{b}2(w+z)^2 + \frac1{2b}(w-z)^2\right) = 0 &\implies b = b^{-1} = 1, \\
        \frac{1}{b+b^{-1}} \left( \sqrt{b}(w+z) + \sqrt{b^{-1}}(w-z)\right)^2 = 0 &\implies w = 0. 
    \end{align*}
    In particular, this means $\WW = 0$ and $\BB$ must be the identity matrix; we note that these two conditions are invariant under conjugation by orthogonal matrices, so even prior to taking the canonical form \emph{\`a la} Remark \ref{rmk:canonical} the matrix $\BB$ must be the identity for all $A(t)$. Returning to the definition of $\BB$ this implies that $A(t) \in \SO(2)$ for all $t$, and $|A(t)| = \sqrt{2}$. 
\end{proof}

\begin{remark}
The convexity of $|A(t)|$ for all geodesics is unique to two dimensions. For $n \geq 3$, consider a geodesic such that, at a particular time $t$, 
\[ \WW = 0,\qquad \BB = \begin{pmatrix} \mu \\ & \nu \\ & & \lambda \\ 
& & & \ddots \\ & & & & \lambda \end{pmatrix}, \qquad\ZZ = \begin{pmatrix} 0 & 1 \\ -1 & 0 \\ && 0 \\ &&& \ddots \\ &&&& 0\end{pmatrix} .\]
We can evaluate $\tr \WW'$ to be
\[ -\frac12 \tr \ZZ\BB \ZZ  + \frac{n\tr \ZZ\BB\ZZ\BB}{2\tr\BB}
= \frac12 (\mu + \nu) -  \frac{n\mu\nu}{\mu + \nu + (n-2)\lambda}.\]
Setting $\mu = \nu$ and $\lambda$ such that $\mu^2 \lambda^{n-2} = 1$, we see that as long as $\lambda$ is sufficiently small the difference
\[ \mu (2 \mu + (n-2)\lambda) - n \mu^2 < 0, \]
showing that there exist geodesics for which $|A(t)|$ is not convex for all time. 
\end{remark}

\section{Linear geodesics}
\label{sect:linear}
As $\Ms(n)$ is an Euclidean space, its geodesics are given by straight lines. 
If a geodesic $A(t)$ of the ambient $\Ms(n)$ space happens to be a curve in the submanifold $\SL(n)$, then necessarily $A(t)$ is also a geodesic of $\SL(n)$. 
We begin this section by classifying all such \emph{linear geodesics} in $\SL(n)$. 
\begin{lemma}\label{lem:linearexponentials}
Let $A(t)$ be a line in $\Ms(n)$. Then $\{A(t)\mid t\in \R\} \subset  \SL(n)$ if and only if $A(t) = B(I+tM)$, where $B \in \SL(n)$ and $M$ is nilpotent.
\end{lemma}
\begin{proof}
If $A(t)$ is a line in $\Ms(n)$, then it can be written in the form $A_0 + t A_1$. It suffices to show
\[ \forall t, A_0 + t A_1 \in \SL(n) \iff A_0 \in \SL(n) \text{ and } A_0^{-1} A_1 \text{ is nilpotent}.\]
First observe that either statement implies that $A_0$ is invertible: on the left evaluate at $t = 0$ to get $A_0 \in \SL(n)$ whereas on the right this is assumed. 
Recall from Corollary \ref{cor:nilpotchar} that $M$ is nilpotent if and only if its characteristic polynomial $\det(sI - M) = s^n$. Since $A_0$ is known to be in $\SL(n)$ starting from either statement, we can compute
\[ \det(A_0 + t A_1) = \det(A_0) \det(I + t A_0^{-1} A_1) = (-t)^n \det\left( \left(-\frac{1}{t}\right) I - A_0^{-1} A_1\right) .\]
So we see that $\det(A_0 + t A_1) = 1$ for all $t$ if and only if $\det(sI - A_0^{-1} A_1) = s^n$. 
\end{proof}

\begin{remark}\label{rmk:char:lin:geo}
    One can check that if $A(t) = A_0 + t A_1$ with $A_0^{-1} A_1$ nilpotent, then $\sff(A',A') \equiv 0$, and so the curve also solves the $\SL(n)$ geodesic equation \eqref{geodesiceq}.
    
    Clearly with the same argument we also see that when $A(t) = (I + tM) B$ with $B\in \SL(n)$ and $M$ nilpotent, we also have a linear geodesic. It turns out that the same geodesic can be equivalently expressed as $B(I + t\tilde{M})$ with $\tilde{M} = B^{-1} M B$ having the same nilpotent index as $M$.
\end{remark}

In the remainder of this section we will explore solutions to the Jacobi equation \eqref{eq:jacobibasic} along linear geodesics. 
To start with, let us define for convenience the symmetric bilinear form 
\begin{equation}
    \mathcal{R}_A(J, K) := \langle R(J, A')A', K\rangle.
\end{equation}
Here $R$ is the Riemann curvature of $\SL(n)$ as given in \eqref{eq:riembilin}. 
If $J(t)$ is a Jacobi field along the geodesic $A(t)$, the Jacobi equation \eqref{eq:jacobibasic} implies that for any $\SL(n)$-tangent vector field $K$ along $A(t)$, 
\[ \langle J'', K\rangle = \mathcal{R}_A(J,K).\]
Since $\sff(A', A') = 0$ for linear geodesics, we see that along such geodesics we have, by \eqref{eq:riembilin}, 
\begin{equation}\label{eq:RAform}
    \mathcal{R}_A(J,K) = \frac{(\tr (A^{-1})'J)(\tr (A^{-1})' K)}{|A^{-1}|^2}. 
\end{equation}
From formula \eqref{eq:RAform} it is clear that $\mathcal{R}_A$ is positive semi-definite. 
An immediate consequence is
\begin{proposition}
    Along linear geodesics, there are no conjugate points. 
\end{proposition}
\begin{proof}
    Let $J$ be a Jacobi field along a linear geodesic $A$. Then
    \[ \langle J, J\rangle'' = 2\langle \frac{D}{dt} J, \frac{D}{dt} J \rangle + 2 \mathcal{R}_{A}(J,J) \geq 0.\]
    This implies that $\langle J, J\rangle$ can vanish at most once along $A$. 
\end{proof}

We next examine the \emph{null space}, the set of $J$ for which $\mathcal{R}_A(J, \bullet) = 0$. 
For $B\in \SL(n)$ and $M$ nilpotent, define the subspace
\begin{equation}\label{eq:def:complinspace}
    \perp_{B,M} := \{ Y\in \Ms(n) \mid \forall 0 \leq k < n, Y \perp B^{-T} (M^{T})^k \}. 
\end{equation}
Since the non-vanishing powers of a nilpotent matrix are linearly independent, the subspace $\perp_{B,M}$ has codimension $N$ in $\Ms(n)$, where $N$ is the nilpotent index of $M$. 
The following technical lemma is convenient.
\begin{lemma}\label{lem:somedecomp:lin}
    Given a linear geodesic $A(t) = B(I + tM)$, the matrix-valued function $A^{-T}(t)$ and all its derivatives are, at every $t$, orthogonal to $\perp_{B,M}$.   
\end{lemma}
\begin{proof}
    Due to the nilpotence of $M$, the function $A^{-1}(t)$ is equal to its degree $n-1$ Taylor polynomial
    \begin{equation}\label{eq:formainvlin}
    A^{-1} B = (I + tM)^{-1} = \sum_{k = 0}^{n-1}(-1)^k(tM)^k.
    \end{equation}
    (The sum terminates at $N-1$ for $M$ whose nilpotent index is $N$.) 
    Therefore $A^{-T}(t)$ is a polynomial in $t$ with coefficients in the orthogonal complement to $\perp_{B,M}$, and the lemma follows. 
\end{proof}

\begin{lemma}\label{lem:comp:lin:space}
    Given a linear geodesic $A(t) = B(I + tM)$ with nilpotent index $N$, the subspace $\perp_{B,M}\subset \Ms(n)$ has the following properties:
    \begin{enumerate}
        \item The set $\perp_{B,M}$ is a subspace of $T_{A(t)} \SL(n)$, with codimension $N-1$. 
        \item The set $\perp_{B,M}$ is contained in the null space of $\mathcal{R}_A$ for every $t$. 
        \item If $J$ is a solution to the Jacobi equation \eqref{eq:jacobibasic} along $A(t)$, such that $J(t_0) = J_0$ and $J'(t_0) = J_1$ both belong to $\perp_{B,M}$, then $J(t) = J_0 + (t-t_0)J_1$ and remains in $\perp_{B,M}$. 
    \end{enumerate}
\end{lemma}
\begin{proof}
    The first property follows from the fact that the normal vector $N(A)$ to $T\SL(n)$ is proportional to $A^{-T}$, which is orthogonal to $\perp_{B,M}$ by Lemma \ref{lem:somedecomp:lin}.  
    
    The second property follows after noting that $\sff(A', J) = 0$ for any $J \in \perp_{B,M}$: this is because $\sff(A',J)$ is proportional to $\langle (A^{-T})', J\rangle$ and this vanishes by Lemma \ref{lem:somedecomp:lin}. 

    For the third property, observe first from the second property, together with \eqref{eq:jacobibasic}, that the second covariant derivative of $J$ vanishes, when $J \in \perp_{B,M}$. Next, since $\sff(A',Y) = 0$ for any $Y\in \perp_{B,M}$, parallel transport of $\perp_{B,M}$ vectors along the curve $A(t)$ with respect to the induced geometry of $\SL(n)$ is identical to parallel transport of the same vector with respect to the ambient geometry of $\Ms(n)$. This implies that $J$ must then solve $J'' = 0$, and hence is a linear function of the given form. 
\end{proof}

When $|A'| = |J| = 1$, and $A' \perp J$, the quantity $\mathcal{R}_A(J,J)$ is equal to the \emph{negative} of the sectional curvature of the plane in $T_A\SL(n)$ spanned by $A'$ and $J$. 
In contrast to the generic situation given by Proposition \ref{prop:curvblowupinf}, we will show that the Riemann curvature along linear geodesics $A(t)$ decay strongly to zero. 

\begin{lemma}
    Let $A(t) = B(I + tM)$ be a linear geodesic of $\SL(n)$, then there exists some constant $C$ depending on $B$ and $M$ such that for any ($\SL(n)$-tangent) vector field $J$ along $A$,  $|\sff(A',J)| = | \tr(A^{-1})'J |/|A^{-1}| \leq C|J| / (1+t^2)$.
\end{lemma}
\begin{proof}
    Let $N$ denote the nilpotent index of $M$. Recall first \eqref{eq:formainvlin}.
     For convenience denote by 
    \[ P:= (A^{-1})' = \sum_{k = 0}^{N-2} (k+1) (-1)^{k+1} t^k M^{k+1} B^{-1}. \]
    For the computation of $\sff(A',J)$, notice that since $J$ is by assumption tangent to $\SL(n)$ and hence orthogonal to $(A^{-1})^T$, we can rewrite
    \begin{equation}\label{eq:proof:sff:lingeo:decay:step}
        \sff(A',J) = - \Big\langle \frac{(A^{-T})'}{|A^{-T}|}, J\Big\rangle = - \Big\langle \Big( \frac{A^{-T}}{|A^{-T}|} \Big)', J\Big\rangle.
    \end{equation}
    Thus we have that 
    \[ \frac{|\sff(A',J)|}{|J|} \leq \left| \left( \frac{A^{-1}}{|A^{-1}|} \right)' \right| \]
    and it suffices to show that this quantity on the right decays like $t^{-2}$.
    
    We can compute directly
    \[ \left( \frac{A^{-1}}{|A^{-1}|} \right)' = \frac{1}{|A^{-1}|^3} \left( P |A^{-1}|^2 - A^{-1} \langle A^{-1}, P\rangle \right). \]
    The leading order term in the denominator is 
    \[
        |A^{-1}|^3  = \left( t^{2N-2} |M^{N-1}B^{-1}|^2 + \ldots \right)^{3/2} = t^{3N - 3} |M^{N-1} B^{-1}|^{3} + \ldots
    \]
    The leading terms of the numerator can be computed as
    \begin{align*}
        P |A^{-1}|^2 & =(-1)^{N-1} \left[ (N-1)  t^{N-2} M^{N-1} B^{-1} - (N-2) t^{N-3} M^{N-2} B^{-1} + \ldots \right] \\
           & \qquad \times \left[ t^{2N-2} |M^{N-1} B^{-1}|^2  - 2 t^{2N - 3} \langle M^{N-1} B^{-1}, M^{N-2} B^{-1}\rangle + \ldots \right]\\
        A^{-1} \langle A^{-1}, P\rangle &= (-1)^{N-1} \left[ t^{N-1} M^{N-1} B^{-1} - t^{N-2} M^{N-2} B^{-1} + \ldots \right] \\
          & \qquad \times [ (N-1) t^{2N-3} |M^{N-1} B^{-1}|^2 \\
          & \qquad \qquad \qquad \qquad- (2N-3) t^{2N-4} \langle M^{N-1} B^{-1}, M^{N-2} B^{-1} \rangle + \ldots ]
    \end{align*}
    and hence
    \begin{multline*}
         P |A^{-1}|^2 - A^{-1} \langle A^{-1}, P\rangle = \\
         (-1)^N t^{3N - 5} \Big[  \langle M^{N-1} B^{-1}, M^{N-2} B^{-1}\rangle M^{N-1} B^{-1} \\
         - |M^{N-1}B^{-1}|^2 M^{N-2} B^{-1} \Big] + \ldots 
    \end{multline*}
    Since $|A^{-1}| \geq n$, the leading order decay rates demonstrated above shows that there exists some constant $C$ such that
    \[ \left| \left( \frac{A^{-1}}{|A^{-1}|} \right)' \right| \leq \frac{C}{1 + t^2} \]
    as desired. (In fact, for every $\epsilon > 0$ one can find $T > 0$ such that 
    \[ \left| \left( \frac{A^{-1}}{|A^{-1}|} \right)' \right| \leq \left( \frac{|M^{N-2}B^{-1}|}{ |M^{N-1} B^{-1}|} + \epsilon\right) \cdot t^{-2}\]
    for all $|t| > T$.)
\end{proof}

\begin{remark}
    The substitution \eqref{eq:proof:sff:lingeo:decay:step} is crucial. While for $\SL(n)$-tangent vectors $J$ the inner products 
    \[ \Big\langle \frac{(A^{-T})'}{|A^{-T}|}, J\Big\rangle \quad  \text{ and } \quad \Big\langle \Big( \frac{A^{-T}}{|A^{-T}|} \Big)', J\Big\rangle\]
    are equal, one can easily check that the norm of $(A^{-T})'/|A^{-T}|$ decays like $t^{-1}$ (instead of $t^{-2}$). This is due to the bulk of this vector pointing in the \emph{normal} direction to $\SL(n)$. In other words, asymptotically most of the change in $A^{-T}$ as one moves along the linear geodesic $A$ is in its length but not its direction. 
\end{remark}

\begin{corollary}\label{cor:lingeo:curvdecay}
    Let $A(t) = B(I+tM)$ be a linear geodesic. Then there exists a constant $C > 0$ depending on $B$ and $M$, such that for any $\SL(n)$-tangent vector fields $J$ and $K$ along $A$, 
    \[ |\mathcal{R}_A(J,K)| \leq \frac{C|J||K|}{(1 + t^2)^2}. \]
\end{corollary}

\begin{theorem}\label{thm:jacobi:rate}
    Let $A(t) = B(I + tM)$ be a linear geodesic in $\SL(n)$.
    Let $J(t)$ be a Jacobi field along $A(t)$, and denote its first covariant derivative along $A$ by $\dot{J} = \frac{D}{dt} J$. Then $|\dot{J}|$ is bounded for all time, and $|J(t)|$ grows at most linearly in $t$. 
\end{theorem}
\begin{proof}
    Since we are interested in asymptotic behavior, we will focus on $t > 1$. 
    Corollary \ref{cor:lingeo:curvdecay} implies
    \[ \frac{d}{dt} \langle \dot{J}, \dot{J}\rangle \leq \frac{2C}{t^4} |\dot{J}| |J|.\]
    We can also compute
    \[
        \frac{d}{dt} \frac{\langle J, J\rangle}{t^{4}} + \frac{4 \langle J, J\rangle}{t^{5}} = \frac{2\langle J, \dot{J}\rangle}{t^4} \leq \frac{2|J||\dot{J}|}{t^{4}}.
    \]
    Define for convenience $\mathfrak{J} = t^{-2} |J| + |\dot{J}|$. We find
    \begin{equation}
        \frac{d}{dt} \mathfrak{J} \leq \frac{C}{t^4} |J| + \frac{|\dot{J}|}{t^2} - \frac{2 |J|}{t^{3}} 
        \leq \max( C, 1 ) t^{-2}\mathfrak{J}.
    \end{equation}
    Hence by Gr\"onwall's Lemma, for $t > 1$,
    \[ \mathfrak{J}(t) \leq \mathfrak{J}(1) \cdot \exp \int_1^t \frac{1 + C}{s^2} ~ds .\]
    Using that $s^{-2}$ is integrable near infinity, this shows that $\mathfrak{J}$ is uniformly bounded, and hence $|\dot{J}|$ is bounded. The basic inequality
    \[ \frac{d}{dt} \langle J, J\rangle \leq 2\langle J, \dot{J}\rangle \implies \frac{d}{dt} |J| \leq |\dot{J}| \]
    then guarantees that $|J|$ grows no faster than linearly. 
\end{proof}

Theorem \ref{thm:jacobi:rate} can be realized explicitly when the nilpotent matrix $M$ has index 2; in this case the equations are explicitly solvable in closed form. 
In this formulation we see not only does the \emph{covariant derivative} $\dot{J}$ of the Jacobi field remain bounded, so does the total derivative $J'$ of the Jacobi field regarded as a vector field along $A$ in $\Ms(n)$.  

\begin{proposition}\label{prop:jacobisolution}
    Along linear solutions $A = I + tM$ where the matrix $M$ satisfies $M^2 = 0$, Jacobi fields have the explicit form
    \[
        J(t) = \left(\frac{|M|^2 t}{n} I + M^T\right) b(t) + K_0 + t K_1,
    \]
    where $K_0, K_1 \in \perp_{I,M}$, and the scalar function $b$ is given by, for free constants $b_0, b_1\in \R$,
    \[ b(t) = \frac{\sqrt{n} b_1}{|M|} \arctan\Big{(} \frac{|M|t}{\sqrt{n}} \Big{)} + b_0.\]
\end{proposition}

\begin{proof}
    Since we are interested in the total derivative $J'$, we will be using the formulation \eqref{eq:jacobieq} of the Jacobi equation. Since $A$ is a linear geodesic, we know that $\sff(A',A') =0$, and the Jacobi equation reduces to
    \[
    J'' - \langle J'', N\rangle N = 0.
    \]
    Since $M^2 = 0$, the normal vector 
    \[ N(t) = \frac{A^{-T}}{|A^{-1}|} = \frac{I - t M^T}{|I - tM^T|}. \]
    As discussed in Lemma \ref{lem:comp:lin:space}, if a solution has initial velocity and position in the codimension 1 (since $M^2 = 0$) subspace $\perp_{I,M}$, such a solution remains in $\perp_{I,M}$ for all time and is linear. Taking advantage of the linearity of the Jacobi equation it remains to consider the portion of $J$ orthogonal to $\perp_{I,M}$.
    
    Such solutions can be expressed as a linear combination
    \begin{equation}
        J(t) = a(t) I + b(t) M^T.
    \end{equation}
    Tangency to $\SL(n)$ requires 
    \begin{equation}
        \langle J, N\rangle = 0 \iff a(t) n = t b(t) |M|^2. 
    \end{equation}
    Immediately we also have
    \begin{gather*}
        a'(t) n = |M|^2 (b(t) + t b'(t)),\\
        a''(t) n = |M|^2(2b'(t) + t b''(t)).
    \end{gather*}
    The Jacobi equation, on the other hand, requires $J''$ to be proportional to the normal vector $N$, which requires
    \begin{equation}
        t a'' + b'' = 0.  
    \end{equation}
    This implies we must have
    \begin{equation}
        \frac{2|M|^2t}{n}  b' + \left(1 + \frac{|M|^2t^2}{n}\right) b'' = 0.
    \end{equation}
    We can explicitly integrate the equation for $b'$ to find
    \[ b'(t) = \frac{  b'(0)}{ 1+ \frac{|M|^2}{n} t^2}. \]
    A second integration yields
    \[ b(t) = \frac{\sqrt{n} b'(0)}{|M|} \arctan\big( \frac{|M|t}{\sqrt{n}} \big) + b(0) \]
    and our claim follows. 
\end{proof}

\begin{remark}
    We showed above that Jacobi fields around linear geodesics have at most linear growth. A large number of these fields actually arise from \emph{explicit} families of nearby geodesics. The Jacobi fields are solutions to a second order ordinary differential equation, for function valued in an $n^2 - 1$ dimensional space, and so enjoys $2n^2 - 2$ degrees of freedom. 
    
    By Lemma \ref{lem:nilpotconj} we can assume without loss of generality that our linear geodesic $B(I + tM)$ is such that $M$ is strictly upper triangular. As all strictly upper triangular matrices are nilpotent, we see that for any strictly upper triangular matrix $T$, and any element $a\in \sl(n)$, the expression
    \[   Be^{sa}(I + t(M + sT)) \]
    gives a one parameter family (by $s$) of linear geodesics. The freedom to choose $a$ and $T$ is $\frac12n(n-1) + (n^2-1)$ dimensional. 
    Their corresponding Jacobi fields are 
    \[ Ba + t BaM + t BT,\]
    and we see that while this family has overlap with the linear Jacobi fields described in Lemma \ref{lem:comp:lin:space}, neither include the other as subsets. 
\end{remark}

\section{Exponential Solutions}
\label{sect:exp}
Consider the curves defined by the matrix exponential of the form 
\[ A(t) = B e^{tC}. \]
Here $B\in \SL(n)$ is the \emph{base point} and $C\in \sl(n)$ is the initial velocity. The curve $A$ therefore lies on $\SL(n)$. 
We ask the question: for what $B, C$ are such curves geodesics under the HS metric? 
 
We find that such geodesics come in two families: they are either pure rotations with $C\in \so(n)$, or linear solutions with $C$ index 2 nilpotent. The first family, which have $|A(t)|$ constant in time, only exists when $n$ is even.

\subsection{Classification of exponential geodesics}
First we verify that there are purely rotating exponential solutions.  

\begin{ex}[Rotational Solutions]
Let $C \in \so(2n)$ and $B \in \SL(2n)$ such that, for some $\kappa\in \R$,
\begin{equation}\label{eq:rot:condition}
    C^2 = \kappa(B^TB)^{-1}.
\end{equation}
We claim that $A(t) = Be^{tC}$ is a geodesic.

First, as $C\in \so(2n)\subset \sl(2n)$, clearly $A(t)\subset \SL(2n)$.
We can compute from our ansatz
\begin{align*}
    A''(t) &= BC^2e^{tC}\\
    &= \kappa B^{-T}e^{tC}.
\end{align*}
On the other hand, $A^{-T} = B^{-T} e^{-tC^T} = B^{-T} e^{tC}$ using that $C\in\so(2n)$ and hence is anti-symmetric. As $A''$ is orthogonal to $\SL(2n)$, the curve $A(t)$ is geodesic. 
\end{ex}

\begin{rem}\label{rem:expgeo}
    Let us examine the condition \eqref{eq:rot:condition} more carefully. 
    \begin{enumerate}
        \item Under the assumption that $C\in \so(2n)$ and hence is anti-symmetric, $C^2$ must be negative semi-definite. On the other hand, since $B\in \SL(2n)$ and is invertible, $(B^TB)^{-1}$ is necessarily \emph{positive definite}. Hence we must have that the proportionality constant $\kappa < 0$.
        \item If $C\in \so(m)$, it is anti-symmetric, and hence normal, and hence diagonalizable. Since $C^2$ is negative semi-definite the non-zero eigenvalues of $C$ are all purely imaginary and come in conjugate pairs. In particular, when $m$ is odd the matrix $C$ must have a kernel. This implies that the equation \eqref{eq:rot:condition} cannot be solved with $C\in \so(m)$ and $B\in \SL(m)$ when $m$ is odd. 
        \item Then $m = 2n$ as we assumed in the example, the structure of $C$ also implies that the singular values of $B$ must come in pairs, and $C$ is formed by the direct sum of rotations of these mutually orthogonal planes. 
        
        More precisely, the singular value decomposition for an arbitrary $B\in \SL(2n)$ yields $B = U\Lambda V$, where $\Lambda$ is diagonal and $U,V\in \SO(2n)$. Equation \eqref{eq:rot:condition} implies we can choose the decomposition such that there exists $0 < \lambda_1 \leq \lambda_2 \leq \cdots \leq \lambda_n$ and $\Lambda = \mathrm{diag}(\lambda_1, \lambda_1, \lambda_2, \lambda_2, \ldots, \lambda_n, \lambda_n)$. 
        
        Denoting by $Z_2$ the $2\times 2$ anti-symmetric matrix $\begin{pmatrix} & -1 \\ 1\end{pmatrix}$, the requirement that $C\in \so(2n)$ in \eqref{eq:rot:condition} means that $VCV^T$ (here $V$ is from the singular value decomposition of $B$) takes the block-diagonal form
        \[ V C V^T = \begin{pmatrix} \omega_1 Z_2 \\ & \omega_2 Z_2 \\ && \ddots \\ &&& \omega_n Z_2 \end{pmatrix}. \]
        The coefficients $\omega_i\in\R$ must further satisfy $(\omega_i)^2 = |\kappa| (\lambda_i)^2$.
        \item Using that $\SO(n)$ acts on $\SL(n)$ as isometries both on the left and on the right, we see that canonically such geodesics $A(t)$ can be described as formed by decomposing $\R^{2n}$ into $n$ pair-wise orthogonal 2-dimensional subspaces. Each of the subspaces is stretched by a factor $\lambda$, and is set to rotate at an angular speed $\omega$ that is inversely proportional to $\lambda$.
        
        Physically the \emph{shape} of the corresponding deformed fluid ball is constant in time, and the Eulerian description of the flow will show this as a steady-state. The condition \eqref{eq:rot:condition} is then the statement that the total speed of every particle on the surface of the fluid ball is constant along the surface.
        This can be regarded as a manifestation of \emph{Bernoulli's principle}; for were the speed not constant there would be a pressure differential along the surface. 
    \end{enumerate}
\end{rem}

As discussed above, the value of $\kappa$ in \eqref{eq:rot:condition} is necessarily negative. It turns out that the limiting case where $\kappa = 0$ also gives rise to a large class of exponential geodesics. 
\begin{ex}[Linear solutions] \label{ex:linsolexp}
Let $C\in \sl(n)$ be such that $C^2 = 0$. Then for any $B \in \SL(n)$, the curve $A(t) = B e^{Ct}$ is geodesic. This follows from Lemma \ref{lem:linearexponentials} after noting that $e^{Ct} = I + Ct$ when $C^2 = 0$. 
\end{ex}

We now show that these two examples cover the possibilities for exponential geodesics of form $A(t) = B e^{tC}$, where $B\in \SL(n)$ and $C\in \sl(n)$. 
From the exponential ansatz, the geodesic equation (\ref{geodesiceq}) gives
\[
    BC^2e^{tC} = \frac{\tr(C^2)}{\tr(B^{-T}e^{-tC^T}e^{-tC}B^{-1})}B^{-T}e^{-tC^T}.
\]
We consider first the case that $\tr(C^2) = 0$. This would imply that $B C^2 e^{tC} = 0$; but under our exponential ansatz both $B$ and $e^{tC}$ are invertible, and hence necessarily $C^2 = 0$. 
This shows that if $\tr(C^2) = 0$ the solution must be one covered by Example \ref{ex:linsolexp}. 

Next we consider the case that $\tr(C^2) \neq 0$, which implies that $C^2$ does not vanish identically. We write $D = (B^TB)^{-1}$ for convenience and rearrange using the invertibility of $B$ and $e^{tC}$, yielding
\begin{equation}\label{classifyingshifted}
    \frac{C^2}{\tr(C^2)} = \frac{D^{-1}e^{-tC^T}e^{-tC}}{\tr(D^{-1}e^{-tC^T}e^{-tC})}.
\end{equation}
Since the left hand side is constant, we have
\[
    0 = \frac{d}{dt}\left[\frac{e^{-tC^T}e^{-tC}}{\tr(D^{-1}e^{-tC^T}e^{-tC})}\right],
\]
which, after evaluating at $t = 0$, yields
\[
   0 = -\frac{C^T + C}{\tr(D^{-1})}+\frac{\tr(D^{-1}(C^T + C))}{(\tr(D^{-1}))^2}.
\]
We know $\tr(D^{-1}) \neq 0$ since $D$ is positive definite and in $\SL(n)$. From this we have
\begin{align}\label{cinson}
    C + C^T = \frac{\tr((C + C^T)D^{-1})}{\tr(D^{-1})}I.
\end{align}
Now, by assumption $C\in \sl(n)$ and has zero trace, and thus the left hand side of (\ref{cinson}) has vanishing trace.
The right hand side of (\ref{cinson}) is \emph{pure trace}, however, and therefore $C + C^T$ must vanish identically, showing that $C$ is anti-symmetric or equivalently $C\in \so(n)$. 
Since $C = -C^T$, we can rewrite equation (\ref{classifyingshifted}) as
\begin{align*}
    C^2 = \frac{\tr(C^2)}{\tr(D^{-1})}D^{-1},
\end{align*}
i.e. $C^2 = \kappa (B^TB)^{-1}$ for some $\kappa \in \R$. By the observations in Remark \ref{rem:expgeo}, such geodesics only exist in even dimensions.

To summarize, we have proved
\begin{theorem}
Let $A(t) = Be^{tC}$, $B \in \SL(n)$, $C \in \sl(n)$. Then $A(t)$ is a geodesic if and only if exactly one of the following holds:
\begin{itemize}
    \item $C^2 = 0$, or
    \item $n$ is even, $C \in \so(n)$, and $C^2 = \kappa (B^TB)^{-1}$ for some $\kappa \in \R$.
\end{itemize}
\end{theorem}

\subsection{Left and right exponentiation yield the same solutions}
Our previous discussion focused on the case where $A(t) = B e^{tC}$ has the exponentiation acting on the right of $B$. 
Similar statements can also be proven for exponentiation on the left of the form $A(t) = e^{tC} B$. How are the two classes related?

In the case $C^2 = 0$, the two classes are the same as already described in Remark \ref{rmk:char:lin:geo}. So we will focus on the case where $C\in \so(n)$. Note first that the left multiplication version of our theorem would require, instead of $C^2 = \kappa B^{-1} B^{-T}$, that $C^2 = \kappa B^{-T} B^{-1}$. 
What we find, however, is that \emph{for a fixed base point $B$} that is compatible with purely rotating exponential solutions, the set of left-rotations and the set of right-rotations are in fact equal. More precisely, we have:
\begin{theorem}
For a fixed $B \in \SL(2n),\ \kappa \in (-\infty,0)$, the sets \[G_R := \{Be^{tC}: C \in \so(2n),\ C^2=\kappa B^{-1}B^{-T}\}\] and \[G_L := \{e^{tC}B: C \in \so(2n),\ C^2=\kappa B^{-T}B^{-1}\}\]
are equal.
\end{theorem}
\begin{proof}
Each element of $G_R$ can be identified by its corresponding $C_R\in\so(2n)$; and each element in $G_L$ can be identified by its corresponding $C_L\in \so(2n)$. It suffices to establish a bijection between these infinitesimal rotations. 

This bijection follows after noticing that for any $C\in \so(2n)$, such that $C^2 = \kappa B^{-1} B^{-T}$ that 
\begin{itemize}
    \item $(BCB^{-1})^2 = BC^2 B^{-1} = \kappa B^{-T} B^{-1}$;
    \item $e^{t BCB^{-1}} B = B e^{tC}$;
    \item and finally $BC B^{-1}$ is also an element of $\so(2n)$. 
\end{itemize}
Only the third bullet point is not entirely obvious: for arbitrary antisymmetric matrix $C$ and arbitrary $\SL(n)$ matrix $B$, it is not the case that $(BC B^{-1})^T = - BCB^{-1}$. Here we refer back to Remark \ref{rem:expgeo}, where we discussed the canonical form of the matrices $B$ and $C$. 
Taking the same notation where $B = U\Lambda V$ in singular value decomposition and $VCV^T$ is block diagonal with $2\times 2$ anti-symmetric blocks, we see that $\Lambda$ commutes with $VCV^{-1}$ and hence
\[ B C B^{-1} = U\Lambda VCV^{-1} \Lambda^{-1} U^{-1} = U (VCV^{-1}) U^{-T} \]
is anti-symmetric. 
\end{proof}

\begin{rem}
    This theorem has a physical interpretation, using that the fluid ball under affine motion at time $t$ is obtained through multiplying the unit ball in $\mathbb{R}^n$ on the left by the deformation $A(t)$. 
    
    For general $B\in \SL(n)$ and $C\in \so(n)$, one can interpret curves of the form $A(t) = Be^{tC}$ as describing a swirling fluid ball that is deformed to fill the shape of a certain ellipsoid, while $A(t) = e^{tC}B$ describes a deformed fluid ball that rotates. 
    The order of operations in the previous sentence reflecting the order with which $B$ and $e^{tC}$ act. 
    In general these two classes of curves are not identical: the fluid balls corresponding to the first class always have an invariant-in-time shape, while those to the second class may have a shape that tumbles over time.
    
    Our theorem shows that geodesic motion (corresponding to solutions of Euler's equation) require that purely rotating solutions to be also purely swirling and vice versa, with the solutions corresponding to certain stationary solutions of Euler's equation in the Eulerian coordinates. 
\end{rem}

\section{Block Diagonal Solutions}
\label{sect:bd}
The purely rotating geodesic solutions of the previous section have, in canonical form, a $2\times 2$ block diagonal decomposition. 
Inspired by those solutions, we examine in this section solutions whose motion is restricted to $2\times 2$ blocks. 
By virtue of Remark \ref{rmk:block-diag}, it is most natural to approach a study of these solutions using the first order formulation \eqref{eq:fo:bb}--\eqref{eq:fo:ww}. 

\begin{notation}
We introduce the following notations for convenience of computation. \begin{enumerate}
    \item When given an $m\times m$ matrix $K_m$ and an $n\times n$ matrix $K_n$, we write $K_m\oplus K_n$ for the $(m+n)\times (m+n)$ block-diagonal matrix $\begin{pmatrix}K_m \\ & K_n\end{pmatrix}$.
    \item We denote by $I_n$ the $n\times n$ identity matrix.
    \item We denote by $K_2$ the $2\times 2$ matrix $\begin{pmatrix} 1 \\ & -1\end{pmatrix}$.
    \item We denote by $Z_2$ the $2\times 2$ matrix $\begin{pmatrix} & -1 \\ 1 \end{pmatrix}$.
    \item We denote by $S_2$ the $2\times 2$ matrix $\begin{pmatrix} & 1 \\ 1\end{pmatrix}$. 
\end{enumerate}
We note that the set $\{I_2, K_2, Z_2, S_2\}$ forms an orthogonal basis of $\Ms(2)$ with the HS inner-product.
\end{notation}
The $2\times 2$ matrices have the following multiplication table (the entries are $RC$, where $R$ is the row label and $C$ the column label):
    \[\begin{array}{c|cccc}
    & I_2 & K_2 & S_2 & Z_2 \\
    \hline
    I_2 & I_2 & K_2 & S_2 & Z_2\\
    K_2 & K_2 & I_2 & - Z_2 & - S_2 \\
    S_2 & S_2 & Z_2 & I_2 & K_2 \\
    Z_2 & Z_2 & S_2 & - K_2 & - I_2
    \end{array}\]
In terms of this notation, the purely rotating solutions, in the canonical form discussed in Remark \ref{rem:expgeo}, has
\begin{equation}
    \begin{aligned}
    \BB &= (\lambda_1)^{-2}I_2 \oplus (\lambda_2)^{-2} I_2 \oplus \cdots \oplus (\lambda_n)^{-2} I_2\\
    \WW &= 0\\
    \ZZ &= \ell_1 Z_2 \oplus \ell_2 Z_2 \oplus \cdots \oplus \ell_n Z_2
    \end{aligned}
\end{equation}
where the conserved angular momenta has components $\ell_i = \pm \sqrt{|\kappa|} \lambda_i$. 

For the remainder of this section, we will focus on block diagonal solutions to \eqref{eq:fo:bb}--\eqref{eq:fo:ww}, subject to the compatibility conditions that 
\begin{itemize}
    \item $\BB$ is positive definite with determinant 1;
    \item $\tr\BB\WW = 0$.
\end{itemize}
We shall pose the ansatz:
\begin{ass}\label{ass:blockdiag}
    We suppose that $\BB, \WW, \ZZ$ have the following decompositions. 
    \begin{itemize}
        \item When $n = 2m$, we assume there are vector-valued functions $b_0, b_1, b_2, w_0, w_1, w_2$ each from $\R$ to $\R^{m}$ and a fixed vector $z\in\R^m$ such that 
        \begin{align*}
        \BB & = \oplus_{i = 1}^m ( b_{0,i} I_2 + b_{1,i} K_2 + b_{2,i} S_2)\\
        \ZZ & = \oplus_{i = 1}^m (z_i Z_2) \\
        \WW & = \oplus_{i = 1}^m (w_{0,i} I_2 + w_{1,i} K_2 + w_{2,i} S_2)
        \end{align*}
        \item When $n = 2m+1$, we assume there are vector-valued functions $b_0, b_1, b_2, w_0, w_1, w_2$ each from $\R$ to $\R^{m}$, a fixed vector $z\in\R^m$, and $b_\infty, w_\infty$ scalar valued functions, such that 
                \begin{align*}
        \BB & = \left[\oplus_{i = 1}^{m} ( b_{0,i} I_2 + b_{1,i} K_2 + b_{2,i} S_2)\right] \oplus b_\infty I_1\\
        \ZZ & = \left[\oplus_{i = 1}^m (z_i Z_2)\right] \oplus 0 I_1\\
        \WW & = \left[\oplus_{i = 1}^m (w_{0,i} I_2 + w_{1,i} K_2 + w_{2,i} S_2)\right] \oplus w_\infty I_1
        \end{align*}
    \end{itemize}
    The vector functions are required to satisfy additionally
        \[ \begin{cases} 
             b_0 \cdot w_0 + b_1\cdot w_1 + b_2 \cdot w_2 = 0, & n = 2m\\
             b_0 \cdot w_0 + b_1 \cdot w_1 + b_2 \cdot w_2 + \frac12 b_\infty w_\infty = 0, & n = 2m+1 
             \end{cases}.\]
\end{ass}
\begin{remark}
	Note that by virtue of Remark \ref{rmk:block-diag}, if the solution satisfies the block diagonal decomposition at any time, it will continue to satisfy the decomposition at all times. Our ansatz above is primarily to make clear the parametrization of such solutions, in terms of the dynamical variables $b_0, b_1, b_2, w_0, w_1, w_2$ (and also $b_\infty, w_\infty$ when the dimension is odd). That the parametrization is consistent is due to the fact that $I_2, K_2$ and $S_2$ form a basis for the set of all symmetric $2\times 2$ matrices. 
\end{remark}

\subsection{Some preliminary computations}\label{ssect:prelim:comp:bd}
Since the block diagonal form is preserved under matrix multiplication, we first perform some preliminary computations for the individual $2\times 2$ blocks. For simplicity of notation we will assume $b_0, b_1, b_2, w_0, w_1, w_2, z$ are scalar valued and drop the vector index. In this case we have 
\[ \BB = b_0 I_2 + b_1 K_2 + b_2 S_2,\quad \WW = w_0 I_2 + w_1 K_2 + w_2 S_2,\quad \ZZ = z Z_2.\]
So 
\begin{multline}
    \WW\BB = (w_0 b_0 + w_1b_1 + w_2b_2)I_2+(w_0 b_1 + w_1 b_0)K_2\\
    + (w_0 b_2 + w_2 b_0)S_2 + (w_2 b_1 - w_1 b_2) Z_2.
\end{multline} 
And
\begin{equation}
    \BB\ZZ = b_0 z Z_2 - b_1 z S_2 + b_2 z K_2.
\end{equation}

Thus
\begin{multline}
    \BB\WW\BB = (w_0 (b_0^2 + b_1^2 + b_2^2) + 2 w_1 b_1 b_0 + 2w_2 b_2 b_0) I_2 \\
    + (2w_0 b_1 b_0 + w_1 (b_0^2 + b_1^2 - b_2^2 ) + 2 w_2 b_2 b_1)  K_2\\
    + (2w_0 b_0 b_2 + w_2(b_0^2- b_1^2 + b_2^2) + 2w_1 b_1 b_2) S_2.
\end{multline}
\begin{multline}
    \WW\BB\WW = (b_0 (w_0^2 + w_1^2 + w_2^2) + 2 w_1 b_1 w_0 + 2w_2 b_2 w_0) I_2 \\
    + (2w_0 w_1 b_0 + b_1 (w_0^2 + w_1^2 - w_2^2 ) + 2 w_2 b_2 w_1)  K_2\\
    + (2w_0 b_0 w_2 + b_2(w_0^2- w_1^2 + w_2^2) + 2w_1 b_1 w_2) S_2.
\end{multline}
\begin{equation}
    \ZZ^T\BB\ZZ = b_0 z^2 I_2 - b_1 z^2 K_2 - b_2 z^2 S_2.
\end{equation}
\begin{multline}
    \ZZ^T\BB \WW + \WW \BB \ZZ = 
    2(w_1 b_2 z - w_2 b_1 z) I_1 \\
    + 2(w_0 b_2 z + w_2 b_0 z)K_2
    + 2(-w_0 b_1 z- w_1 b_0 z)S_2.
\end{multline}
Also, we have (these are the traces of the $2\times 2$ blocks)
\begin{equation}
    \tr \BB\ZZ\BB\ZZ = 2z^2 (b_1^2 + b_2^2 - b_0^2)
\end{equation}
and 
\begin{equation}
    \tr \WW\BB\WW\BB =  4(w_0 b_0 + w_1b_1 + w_2b_2)^2
    + 2 (w_1^2 + w_2^2 - w_0^2 )(b_0^2 - b_1^2 - b_2^2). 
\end{equation}

Since the coupling in \eqref{eq:fo:bb}--\eqref{eq:fo:ww} between different blocks when $\BB, \ZZ, \WW$ are block diagonal is only through the diagonal part of \eqref{eq:fo:ww}, the above computations show that if $\BB$ and $\WW$ are both pure-trace on a $2\times 2$ block for a solution satisfying Assumption 5.2, then they will remain so for all time (as the trace-free components of $\BB'$ and $\WW'$ both vanish on that block). 
We are particularly interested in the case where all blocks are pure-trace:
\begin{proposition}\label{prop:pure:diag}
 Under Assumption \ref{ass:blockdiag}, if $w_1 = w_2 = b_1 = b_2 =0$ at some time $t$, then they vanish for all time. 
\end{proposition}

Secondly, we also find that the total conserved energy $\frac12\tr (\ZZ^T + \WW)\BB(\WW + \ZZ)$ takes the following forms:
when $n = 2m$ is even, 
\begin{multline}\label{eq:energy:bd:even}
        \frac12\tr (\ZZ^T + \WW)\BB(\WW + \ZZ) = \sum_{i = 1}^m b_{0,i} (w_{0,i}^2 + w_{1,i}^2 + w_{2,i}^2 + z_i^2) \\ + 2 w_{0,i}w_{1,i}b_{1,i} - 2 w_{2,i} b_{1,i} z_i + 2 w_{0,i} w_{2,i} b_{2,i} + 2 w_{1,i} b_{2,i} z_i;
\end{multline}
when $n = 2m+1$ is odd,
\begin{multline}\label{eq:energy:bd:odd}
        \frac12\tr (\ZZ^T + \WW)\BB(\WW + \ZZ) = \frac12 b_{\infty} w_\infty^2 + \sum_{i = 1}^m b_{0,i} (w_{0,i}^2 + w_{1,i}^2 + w_{2,i}^2 + z_i^2) \\ + 2 w_{0,i}w_{1,i}b_{1,i} - 2 w_{2,i} b_{1,i} z_i + 2 w_{0,i} w_{2,i} b_{2,i} + 2 w_{1,i} b_{2,i} z_i. 
\end{multline}

\subsection{Bounds on \texorpdfstring{$\BB$}{beta} via conserved energy}
The interplay between the conservation of energy and conservation of angular momenta provides some partial bounds on the size of $\BB$. 

Given vectors $u,v\in \R^2$, we have the decomposition 
\[ |u|^2 |v|^2 = |u\cdot v|^2 + |u\wedge v|^2.\]
So
\begin{align*}
    \big| w_{0,i} ( w_{1,i} b_{1,i} + w_{2,i} b_{2,i}) &+ z_i (w_{1,i} b_{2,i} - w_{2,i} b_{1,i}) \big| \\
    & \leq \sqrt{ w_{0,i}^2 + z_i^2} \sqrt{ |w_{1,i} b_{1,i} + w_{2,i} b_{2,i}|^2 + |w_{1,i} b_{2,i} - w_{2,i} b_{1,i}|^2} \\
    & = \sqrt{ w_{0,i}^2 + z_i^2} \sqrt{w_{1,i}^2 + w_{2,i}^2} \sqrt{b_{1,i}^2 + b_{2,i}^2} \\
    & \leq \frac12 \sqrt{b_{1,i}^2 + b_{2,i}^2} (w_{0,i}^2 + w_{1,i}^2 + w_{2,i}^2 + z_i^2).
\end{align*}
Combined with \eqref{eq:energy:bd:even} and \eqref{eq:energy:bd:odd} this gives the following coercivity properties on the conserved energy.
\begin{proposition}\label{prop:energy:am:bnd}
    Let $\BB, \WW,\ZZ$ describe a geodesic with the block diagonal decomposition given by Assumption \ref{ass:blockdiag}, then there exists a constant $E$ (the total conserved energy) such that for every $i = 1, \ldots, m$, 
    \[ z_i^2 \left( b_{0,i} - \sqrt{b_{1,i}^2 + b_{2,i}^2} \right) \leq E.\]
\end{proposition}
In particular, this shows that for each of the $2\times 2$ blocks of $\BB$ corresponding to a non-vanishing angular momentum $z_i$, \emph{at most one of the two eigenvalues} can be of size $ \sqrt{E} / z_i$ or larger. 

\begin{theorem}\label{thm:boundedness}
    Let $n = 2m$ be even. Suppose $\BB, \WW, \ZZ$ corresponds to a geodesic satisfying the block diagonal Assumption \ref{ass:blockdiag}, such that $b_1 = w_1 = b_2 = w_2 = 0$ at some time. Suppose further that $z_i \neq 0$ for any $i$. Then the motion described by $\BB, \WW, \ZZ$ is bounded. 
\end{theorem}
\begin{proof}
    By Proposition \ref{prop:pure:diag} we have that $b_1 = w_1 = b_2 = w_2 = 0$ for all times. Proposition \ref{prop:energy:am:bnd} implies that for every $i = 1,\ldots, m$ we have 
    \[ b_{0,i} \leq \frac{E}{z_i^2} \]
    and so each individual $b_{0,i}$ is bounded. However, since $\BB$ has determinant 1, we must also have $\prod_{i = 1}^m b_{0,i} = 1$. If any individual $b_{0,i}$ were to approach 0, the product constraint would require some other $b_{0,j}$ to grow unboundedly. Therefore we conclude that $b_{0,i}$ must each individually be bounded above and below, and the motion is bounded. 
\end{proof}

\subsection{Periodically Pulsating Solutions}
A special case of the bounded solutions described in Theorem \ref{thm:boundedness} are those that those where, at any time, $b_{0,i}$ takes one of two values. 
\begin{proposition}\label{prop:pulse}
 Let $\BB, \WW, \ZZ$ be as in Theorem \ref{thm:boundedness}. Suppose further that, at some time $t$, there exists $1 \leq m_0 < m$ such that 
 \[ 0 \neq |z_1| = \cdots = |z_{m_0}|, \qquad |z_{m_0 + 1}| = \cdots = |z_m| \neq 0,\]
 and
 \[ b_{0,1} = \cdots = b_{0,m_0}, \qquad  b_{0,m_0+1} = \cdots = b_{0,m}, \]
 and
 \[ w_{0,1} = \cdots = w_{0,m_0}, \qquad  w_{0,m_0+1} = \cdots = w_{0,m}.\]
 Then these conditions are preserved for all time and the solution is periodic in time. 
\end{proposition}
\begin{proof}
    The preservation of the conditions follow from the computations in Section \ref{ssect:prelim:comp:bd}, whence we see that derivatives of $b_{0,i}$ and $w_{0,i}$ for $i = 1, \ldots, m_0$ are all equal, and similarly for $i = m_0 + 1, \ldots, m$. 
    
    The incompressibility assumption $\det \BB = 1$ requires that 
    \[ b_{0,1}^{m_0} b_{0,m}^{m - m_0} = 1 \]
    so we can write
    \[ b_{0,1} = e^{\lambda / m_0}, \qquad b_{0,m} = e^{-\lambda/(m-m_0)} \]
    for some time-dependent $\lambda$. The condition that $\tr \BB\WW = 0$ shows that we can write 
    \[ w_{0,1} = \frac{1}{m_0} e^{-\lambda/m_0} v, \qquad w_{0,m} = -\frac{1}{m-m_0} e^{\lambda / (m-m_0)}v \]
    for some time-dependent $v$. 
    
    We thus see that our equations of motion reduce to the Hamiltonian dynamics on a two-dimensional phase space for $(\lambda, v)$ with Hamiltonian being
    \[ H = m_0 e^{\lambda / m_0} z_1^2 + (m-m_0) e^{-\lambda/(m-m_0)} z_m^2 + \left(\frac{1}{m_0} e^{-\lambda/m_0}  + \frac{1}{m-m_0} e^{\lambda / (m-m_0)}\right) v^2. \] 
    (Notice that $z_1$ and $z_m$ are considered to be fixed constants.) Therefore by Liouville's theorem \cite{ArnoldMechanics}*{pg. 272}, the solution is periodic in time provided we can show that the level sets of $H$ are compact. 
    
    That the level sets of $H$ is compact follows from $H$ being proper. This is immediate after noting that, since each individual term in the definition of $H$ is non-negative, we have must have each term is individually bounded by $H$, which gives
    \[
    - (m-m_0) \ln\left( \frac{H}{(m-m_0)z_m^2}\right)    \leq \lambda \leq m_0 \ln \left( \frac{H}{m_0 z_1^2} \right) \]
    and
    \[ v^2 \leq \max(m_0, m-m_0) H.\]
    
    (We note in passing that the only critical point of $H$ occurs at $(\lambda_0,0)$, where $\lambda_0$ is the unique solution to 
    \[ e^{\lambda_0 / m_0} z_1^2 = e^{-\lambda_0 / (m-m_0)} z_m^2. \]
    This critical point is, of course, elliptic. Combined with the fact that $H$ is proper this gives another proof that does not depend on Liouville's theorem that the level sets are topologically circles and the dynamics is periodic.)
\end{proof}

\subsection{Higher dimensional swirling and shear flows}
In Proposition \ref{prop:pulse}, if we drop the assumption that neither $z_1$ nor $z_m$ vanish, then it turns out the motion is necessarily non-compact. 
These motions turn out to be generalizations of the swirling and shear flows analyzed by Sideris in \cite{AffineFluid3D}*{Sect. 6}, which corresponds to $n = 3$ and $m_0 = 1$ in the theorem below. 

\begin{remark}
    The flows described in Theorem \ref{thm:swfl} allow $n$ to be either odd or even. This is in contrast to the results of Theorem \ref{thm:boundedness} and Proposition \ref{prop:pulse}. The requirement that $2 \leq 2m_0 < n$ however restricts $n > 2$; so such flows do not occur in dimension 2. 
\end{remark}

\begin{theorem}\label{thm:swfl}
    Let $m_0 \geq 1$ be such that $2m_0 < n$. 
    Consider solutions to the geodesic flow on $\SL(n)$ as described by \eqref{eq:fo:bb}--\eqref{eq:fo:ww} with initial data of the form (which satisfies Assumption \ref{ass:blockdiag})
    \begin{align*}
        \BB(0) &= (\beta_0)^{1/(2m_0)} I_{2m_0} \oplus (\beta_0)^{-1/(n - 2m_0)} I_{n - 2m_0}, \\
        \ZZ(0) &= \left( \oplus_{i = 1}^{m_0} z_0 Z_2 \right) \oplus 0 I_{n-2m_0}, \\
        \WW(0) &= \left(\frac{1}{2m_0} v (\beta_0)^{-1/(2m_0)} I_{2m_0} \right) \oplus\left(- \frac{1}{n - 2m_0} v (\beta_0)^{1/(n-2m_0)} I_{n-2m_0} \right).
    \end{align*}
    Then the solutions preserve the block diagonal form with 
    \[ \BB(t) = \left(e^{b(t)/(2m_0)} I_{2m_0} \right)\oplus \left(e^{-b(t)/(n-2m_0)} I_{n-2m_0}\right) \]
    with $b(t)$ an unbounded function. Furthermore, 
    \begin{itemize}
        \item if the constant $z_0 \neq 0$, then $\lim_{t \to \pm\infty} b(t) = -\infty$;
        \item if the constant $z_0 = 0$, then $b(t)$ is strictly monotonic and the image of $b$ contains the entire real line. 
    \end{itemize} 
\end{theorem}
\begin{proof}
    The the block diagonal form is preserved follows directly from combining \eqref{eq:fo:bb} and \eqref{eq:fo:ww} with the computations in Section \ref{ssect:prelim:comp:bd}. Similarly to the computations in the proof of Proposition \ref{prop:pulse}, we find the solution can be written with 
    \[ \WW(t) = \left(\frac{v(t)}{2m_0} e^{-b(t)/(2m_0)} I_{2m_0}\right) \oplus \left( - \frac{v(t)}{n - 2m_0} e^{b(t)/(n-2m_0)} I_{n - 2m_0} \right).\]
    The dynamics is therefore that of a Hamiltonian dynamics on the two dimensional phase space (noting that the angular momentum is conserved) $(b,v)$ with Hamiltonian
    \begin{equation}\label{eq:pf:swfl:ham}
        H(b,v) = m_0 z_0^2 e^{b/(2m_0)} + \left(\frac{1}{4m_0} e^{-b/(2m_0)} + \frac{1}{2n - 4m_0} e^{b/(n-2m_0)}\right) v^2.
    \end{equation}
    
    We treat first the case $z_0 \neq 0$. 
    We see that $H$ has no critical points ($\partial_v H = 0$ requires $v = 0$, and when $v = 0$ we see $\partial_b H \neq 0$). 
    Therefore orbits of the Hamiltonian flow must be non-compact in phase space. 
    The conservation of energy then implies that $b(t) \leq  2m_0  \ln[ H / (m_0 z_0^2)]$ is bounded above. 
    Next, writing $p = n / (2m_0)$ and $q = n/(n-2m_0)$, we have
    \[ \frac{1}{4m_0} e^{-b/(2m_0)} + \frac{1}{2n - 4m_0} e^{b/(n-2m_0)} = \frac{p}{2n} \left( e^{-b/n}\right)^{p} + \frac{q}{2n} \left( e^{b/n}\right)^q .\]
    So Young's inequality for products gives
    \[ \frac{ p^{2/p} q^{2/q}}{2n} \leq \frac{1}{4m_0} e^{-b/(2m_0)} + \frac{1}{2n - 4m_0} e^{b/(n-2m_0)} \]
    and we see that conservation of energy implies 
    \[ v(t)^2 \leq \frac{2n H}{p^{2/p} q^{2/q}} \]
    so that $v(t)$ is bounded. 
    Therefore the non-compact orbits must satisfy $\lim_{t\to \pm\infty} b(t) = - \infty$. 
    
    Next we treat the case $z_0 = 0$. 
    We see that here $dH$ vanishes if and only if $v = 0$, where $H = 0$.  These corresponds to static solutions to the geodesic flow. 
    Away from this set (where the conserved energy does not vanish), we see that $\partial_v H \neq 0$, this shows that the level sets of $H$ must form graphs over the $b$ axis in the $(b,v)$ plane, and hence the image of $b$ covers the entire real line for the geodesic flow. Since the conservation of energy forces $v$ to be signed in this case, the equations of motion also forces $b$ to be monotonic. 
\end{proof}

\begin{rem}
    In the case $z_0 = 0$, the non-compactness of the geodesics can also be derived from Theorem \ref{thm:suff:cond:unbnd}.
\end{rem}

\begin{rem}\label{ren:asympt:swfl}
    The computations in the proof of Theorem \ref{thm:swfl} also yield asymptotics for the functions. 
    \begin{enumerate}
    \item We first treat the case $z_0 \neq 0$. First notice that by the equations of motion \eqref{eq:fo:bb} we have that
    \[ b'(t) = - v(t).\]
    If we write $B(t) = e^{-b(t)/(4m_0)}$, we can rewrite the conservation of energy \eqref{eq:pf:swfl:ham} as
    \[ H = \frac{m_0 z_0^2}{B(t)^2} + \left( 1 + \frac{2m_0}{(n - 2m_0) B^{2 + \frac{4m_0}{n - 2m_0}}}\right) 4m_0 B'(t)^2.\]
    This shows that $B'(t)$ is bounded (in fact asymptotically constant), and hence $B(t)$ grows linearly in time. So we can rewrite as
    \[ B'(t)^2 = \frac{H}{4m_0} - \frac{z_0^2}{4B(t)^2} + O(t^{-2 - 4m_0 / (n - 2m_0)}). \] 
    Integrating this we find asymptotically (as $t\to\pm\infty$)
    \begin{gather}
        |B'(t)| = \frac12 \sqrt{\frac{H}{m_0}} - 2 z_0^2 \left( \frac{H}{m_0}\right)^{-3/2} t^{-2} + O(|t|^{-2 - 4m_0 / (n-2m_0)}), \\
        B(t) = \frac12 \sqrt{\frac{H}{m_0}} |t| + 2 z_0^2 \left( \frac{H}{m_0}\right)^{-3/2} |t|^{-1} + O(|t|^{-1 - 4m_0 / (n-2m_0)}).
    \end{gather}
    \item In the case $z_0 = 0$, we shall assume, without loss of generality, that $b'(t) = - v(t) > 0$ (else we reverse the time orientation). In this case we can write $B_-(t) = e^{- b(t)/(4m_0)}$ and $B_+(t) = e^{b(t) / (2n - 4m_0)}$. And we have
    \begin{multline*} 4m_0 (B_-'(t))^2 \left( 1 + \frac{2m_0}{(n - 2m_0) B_-(t)^{2n/(n - 2m_0)}}\right) = H \\
     = (2n - 4m_0) (B_+'(t))^2 \left( 1 + \frac{n - 2m_0}{2m_0 B_+(t)^{n/m_0}} \right).
     \end{multline*}
     Integrating we find that
     \begin{align}
         t\to -\infty &&  B_-(t) &= \sqrt{\frac{H}{4m_0}} |t| + o(|t|^{-1}), \\
         t \to +\infty && B_+(t) &= \sqrt{\frac{H}{2n-4m_0}} |t| + o(|t|^{-1}).
    \end{align}
    \end{enumerate}
    These are in agreement with the asymptotics found in \cite{AffineFluid3D} when $n = 3$ and $m_0 = 1$. 
\end{rem}

A crucial point illustrated by these solutions is that, \emph{in even dimensions} there exists severely unstable non-compact solutions to the geodesic flow. 
Whereas for the linear flows we were able to show some degree of linear stability via our analysis of the Jacobi equation in Theorem \ref{thm:jacobi:rate}, the flows described by Theorem \ref{thm:swfl} is unstable under arbitrarily small perturbations. 
This can be seen by noting that, in even dimensions, there exists arbitrarily small initial data perturbations that makes non-zero all the principal angular momenta, thereby bringing the perturbed solution into the regime covered by Proposition \ref{prop:pulse}. 
And hence there exists arbitrarily small initial data perturbations of these non-compact solutions, whose corresponding solutions are bounded (in fact periodic). 

\begin{remark}
	It is worth noting that this instability is \emph{not} related to the Taylor sign condition. As already mentioned in Remark \ref{rmk:taylorstab}, the Taylor sign condition is not operative for the ODE analysis. 
	In our present case, we can in fact demonstrate that the Taylor sign condition holds asymptotically for the swirling and shear flow using the computations above. 

	By definition the function $v(t)$ is proportional to the log derivative of $B(t)$ (or $B_{\pm}(t)$) in the asymptotic analysis of Remark \ref{ren:asympt:swfl}, and hence is of size $O(1/|t|)$. The second fundamental form itself is proportional to $\tr(\WW\BB\WW\BB) + \tr(\ZZ\BB\ZZ\BB)$. The first term of which is positive and proportional to $|v|^2$. The second term vanishes when $z_0 = 0$; and when $z_0 \neq 0$ it is negative and proportional to $B^{-4}$. 
	This shows that for all sufficiently large $t$, the Taylor sign condition is satisfied (the second fundamental form is positive). 
\end{remark}

\subsection{Numerics}
To illustrate the types of behavior block-diagonal solutions can exhibit, we show in the three figures below results of numerical simulations for solutions under the structural Assumption \ref{ass:blockdiag}. 
In all three cases $n = 6$, and we further assumed that the vectors  $b_1 = b_2 = w_1 = w_2$ vanish identically, as in Proposition \ref{prop:pure:diag}. 
The horizontal axis of the plot represents time. 
The three curves in the plot represent the lengths of the semi-axes of the fluid ball; in other words, their values are the inverse square roots of the three components of the vector $b_0$. The initial values for the simulations, in terms of the vectors $b_0, w_0$, and $z$, are included in the figure captions.  

\begin{figure}[ht!]
        \centering
        \includegraphics[scale = .5]{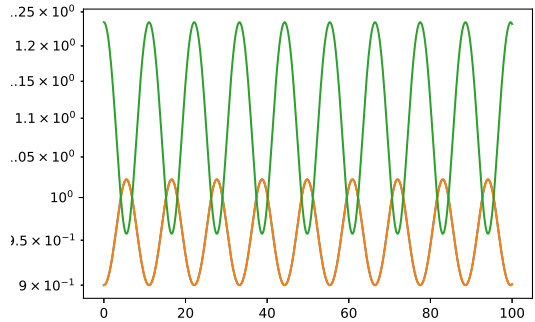}
        \caption{Initial conditions: $z_1 b_{0,1} = 0.5,\ z_2  b_{0,2} = 0.5,\ z_3 b_{0,3} = 0.3,\ w_{0,1} = w_{0,2}  = 0,\ b_{0,1} = b_{0,2} = 0.9^{-2}$}\label{periodic} 
\end{figure}

Figure \ref{periodic} exhibits the type of periodic pulsating behavior described by Proposition \ref{prop:pulse}; here the value of $m_0 = 2$. (The blue and green curves overlap, as the top two diagonal blocks in $\BB$ are identical.)

\begin{figure}[ht!]
        \centering
        \includegraphics[scale = .5]{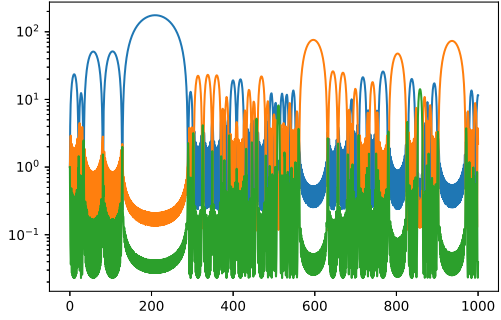}
        \caption{Initial conditions: $z_1 b_{0,1} = 1,\ z_2 b_{0,2} = 0.5,\ z_3 b_{0,3} = 0.1,\ w_{0,1} b_{0,1} = 3,\ w_{0,2}=0,\ b_{0,1} = b_{0,2} = 1$}\label{not periodic} 
\end{figure}

Figure \ref{not periodic} exhibits the generic aperiodic behavior expected of higher dimensional block-diagonal solutions. 
Note that due to the non-vanshing of all three principal angular momenta, by Theorem \ref{thm:boundedness} the solution remains bounded. The Hamiltonian system however has a 4 dimensional phase space and chaotic behavior is admissible; in fact, such seemingly aperiodic behavior seems generic based on numerical simulation. 

\begin{figure}[ht!]
        \centering
        \includegraphics[scale = .5]{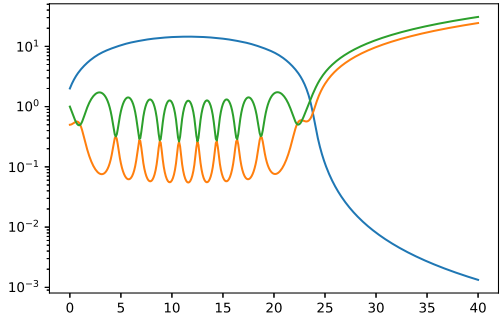}
        \caption{Initial conditions: $z_1 = 0,\ z_2 b_{0,2} = 0.5,\ z_3 b_{0,3} = 0.6,\ w_{0,1} b_{0,1} = 1,\ w_{0,2}=0,\ b_{0,1} = 1/4,\ b_{0,2} = 4$}\label{Initially oscillating} 
\end{figure}

Figure \ref{Initially oscillating} exhibits the behavior when one of the three principal angular momenta vanishes. By Proposition \ref{prop:energy:am:bnd} we see that $b_{0,2}$ and $b_{0,3}$ will be bounded. However the vanishing $z_1$ allows $b_{0,1}$ to grow unboundedly; this manifests numerically as the blue curve decaying to zero as $t\to \pm\infty$. 

In fact, we expect that, under the purely diagonal assumption $w_1 = w_2 = b_1 = b_2 = 0$,
\begin{itemize}
    \item if $n$ is even and exactly one of the components of $z$ vanishes, then as $t\to\pm\infty$ the corresponding component of $b_0$ grows unboundedly; 
    \item if $n$ is odd and none of the components of $z$ vanish, then $b_\infty$ grows unboundedly as $t \to \pm\infty$. 
\end{itemize}
We have only been able to prove this in the special setting of Theorem \ref{thm:swfl}; the analysis in the general case is complicated by the fact that, as seen in Figure \ref{Initially oscillating}, there may be an initial era with oscillatory behavior.  

\FloatBarrier

\section{Future Directions}
Here we list some further questions concerning the geometry of $\SL(n)$ with the HS metric that we think deserve further investigation. 

Based on our results, we conjecture that:
\begin{conjecture}
If $n$ is odd, then $\SL(n)$ has no bounded geodesics.
\end{conjecture}
Even within the class of block-diagonal solutions considered in Section \ref{sect:bd}, this conjecture is still open. 
The relevant results are Theorem \ref{thm:swfl} which settles the block-diagonal case in $n = 3$, but is only a partial result providing sufficient conditions for non-compact geodesics when $n \geq5$; and the fact that when $n$ is odd there can be no analogue of Theorem \ref{thm:boundedness}. 

Another conjecture that we have is:
\begin{conjecture}
Shifted exponential solutions of the form $Be^{tC}$ are the only geodesics with constant norm (as a curve in the vector space $\Ms(n)$ with the HS norm).
\end{conjecture}
This is based on the physical intuition that fluid flows of this type require careful balancing between the angular momenta of the flow and the ``shape'' of the resulting fluid ``ball''. 
In the case $n = 2$ this is known to be true either by the classification in \cite{AffineFluid2D}, or by our Theorem \ref{thm:2dbndedorcstnorm}.
Note that within the class of pulsating solutions considered in Proposition \ref{prop:pulse} this is also true: constant norm fixes the trace of $\BB^{-1}$, and the requirement $\det(\BB) = 1$ and the diagonal ansatz imply that the principal values of $\BB$ are constant in time. 
However, even in the block diagonal case, when we allow three or more distinct principal values for $\BB$ the same argument no longer holds. 

Finally, in view of the potentially complicated geometry near infinity of $\SL(n)$ given by Proposition \ref{prop:curvblowupinf}, it is interesting to ask the following:
\begin{question}
Does there exist an unbounded geodesic that is not asymptotically linear; that is to say, that $A''(0) \not\to 0$ as $t\to \pm\infty$?
\end{question}
By virtue of Proposition \ref{prop:poscurv:unbnd}, flows of this sort would require $\sff(A',A')$ to alternate between signs infinitely often as $t\to \pm\infty$. 
The existence of the solutions of Theorem \ref{thm:swfl}, which are dynamically unstable and non-compact when the dimension is even, hints at this possibility. 
This is reinforced by our numerical experiments, but a theoretical proof is wanting.

\end{document}